\documentclass[oneside]{amsart}
\usepackage{amsmath,amsfonts,amssymb,upgreek, amscd, mathtools}
\usepackage[breaklinks,hidelinks]{hyperref}

\usepackage{tikz}
\usepackage{enumerate}
\usetikzlibrary{calc,decorations.markings}
\usepackage{tikz-cd}
\usepackage{graphicx}

\theoremstyle{plain}

\newtheorem{thm}{Theorem}[section]
\newtheorem{lemma}[thm]{Lemma}
\newtheorem{prop}[thm]{Proposition}

\newtheorem{cor}[thm]{Corollary}

\newtheorem{example}[thm]{Example}

\newtheorem{remark}[thm]{Remark}

\theoremstyle{definition}

\newcommand{\C}{\operatorname{C} }

\newcommand{\Ho}{\operatorname{H} }

\newcommand{\Z}{\operatorname{Z} }

\newcommand{\Fix}{\operatorname{Fix}}

\newcommand{\Id}{\operatorname{id}}

\newcommand{\gen}[1]{\left < #1 \right >}
\newcommand{\Aut}{\operatorname{Aut} }

\newcommand{\Ext}{\operatorname{Ext} }

\newcommand{\Hom}{\operatorname{Hom} }

\newcommand{\Ker}{\operatorname{ker} }

\newcommand{\Ann}{\operatorname{Ann} }
\newcommand{\Soc}{\operatorname{Soc} }

\numberwithin{equation}{section}
\newcommand{\op}{\operatorname{op} }
\newcommand{\Diag}{\operatorname{Diag} }

\newcommand{\Lam}[1]{\Lambda_{#1}}

\newcommand{\ZSb}{\mathrm{Z}^2_{\mathrm{Sb}}}

\begin{document}

\title{Skew brace extensions, second cohomology and complements}

\author{Nishant Rathee$^1$}
\address{$1.$ Department of Mathematics, National Institute of Technology, Tiruchirappalli-620015, India}
\email{nishant@nitt.edu, monurathee2@gmail.com}		

\author{Manoj K. Yadav$^{2,3}$}
\address{$2.$  Harish-Chandra Research Institute, Chhatnag Road, Jhunsi, Prayagraj-211 019, India.}
\address{$3.$ Homi Bhabha National Institute, Training School Complex, Anushakti Nagar, Mumbai 400 094, India}
\email{myadav@hri.res.in}



\begin{abstract}
We study extensions and the second cohomology groups of skew left braces via the associated natural semi-direct product groups. Let $0 \to I \to E \to H \to 0$ be a skew brace extension with $I$ an abelian group viewed as a trivial brace, and $\Lambda_H$ denote the natural semi-direct product associated with the skew left brace $H$. We establish a group homomorphism from $\Ho_{Sb}^2(H, I)$, the second cohomology group of braces, into $\Ho_{Gp}^2(\Lambda_H, I \times I)$, the second cohomology group of groups with the induced actions, which turns out to be an embedding when $I \le \Soc(E)$, the socle of $E$. In particular, the Schur multiplier of a finite skew left brace $H$ embeds into the direct square of the Schur multiplier of the group $\Lambda_H$. An analogue of the Schur--Zassenhaus theorem is established for skew left braces in several specific cases. We introduce a concept called minimal extensions (which stay at the extreme end of split extensions) of skew left braces and derive many fundamental results. Several reduction results for split extensions of finite skew left braces by abelian groups (viewed as trivial left braces) are obtained.
\end{abstract}

\maketitle
	
\section{Introduction}

The cohomology of left braces and linear cycle sets with trivial coefficients was introduced by Lebed and Vendramin \cite{LV16}. Bardakov and Singh studied central extensions and automorphisms of linear cycle sets \cite{BS21}. Building on Bachiller's work on extensions \cite{DB18}, the authors developed cohomology and extension theory for skew left braces in \cite{RY24}. In this paper we study these extensions through associated semidirect product groups. This comparison leads to results on second cohomology, complements, minimal extensions and Frattini sub-skew braces, as well as criteria for splitting extensions.

A skew left brace is a set $A$ with two group operations, $+$ and $\circ$, satisfying
\[
a\circ(b+c)=(a\circ b)-a+(a\circ c)
\qquad(a,b,c\in A).
\]
The groups $(A,+)$ and $(A,\circ)$ are called its additive and multiplicative groups. The lambda map
\[
\lambda\colon(A,\circ)\longrightarrow\Aut(A,+),
\qquad \lambda_a(b)=-a+(a\circ b),
\]
is a group homomorphism \cite{GV17}, and determines the associated group
\[
\Lambda_A=(A,+)\rtimes_\lambda(A,\circ).
\]
A skew brace extension $0\to I\to E\to H\to0$ induces a group extension $0\to\Lambda_I\to\Lambda_E\to\Lambda_H\to0$. When $I$ is an abelian group viewed as a trivial brace, $\Lambda_I=I\times I$, so this construction connects brace extensions with abelian group extensions.

Fix a good triplet of actions $(\nu,\mu,\sigma)$ of $H$ on an abelian group $I$, as defined in Section~\ref{sec-preliminaries}. Write $\Ho^2_{Sb}(H,I)$ for the corresponding second brace cohomology group and $\Ho^2_{Gp}(\Lambda_H,I\times I)$ for group cohomology with the induced action described in Section~\ref{sec-cohomology}. Our first result combines Theorem~\ref{Psi-hom} and Corollary~\ref{useful-cor}.
\begin{thm}
Let a skew left brace $H$ act on an abelian group $I$ via a good triplet $(\nu,\mu,\sigma)$. There is a group homomorphism
\[
\Psi\colon\Ho^2_{Sb}(H,I)\longrightarrow\Ho^2_{Gp}(\Lambda_H,I\times I).
\]
If $\mu$ is trivial and $\nu_h^{-1}=\sigma_h$ for every $h\in H$, then $\Psi$ is injective.
\end{thm}

The hypotheses for injectivity are equivalent to $I\leq\Soc(E)$ for an extension $E$ inducing the given actions. For finite $H$, taking $I=\mathbb C^*$ with trivial actions gives an embedding of the brace Schur multiplier into the direct square of the group Schur multiplier of $\Lambda_H$. This recovers the bound $\exp M_b(H)\mid |H|^2$ of \cite[Theorem~2.2]{LV24}. We also obtain bounds for trivial and almost trivial skew braces and for braces of odd prime-power order.

Our second comparison concerns the opposite skew brace $H^{\op}=(H,+^{\op},\circ)$, where $a+^{\op}b=b+a$.
\begin{thm}
Let $H$ act on a trivial brace $I$ via a good triplet $(\nu,\mu,\sigma)$. Equip $H^{\op}$ with the induced triplet
\[
\nu_h^{\op}=\mu_{-h}\nu_h,\qquad
\mu_h^{\op}=\mu_{-h},\qquad
\sigma_h^{\op}=\sigma_h.
\]
Then $\Ho^2_{Sb}(H,I)\cong\Ho^2_{Sb}(H^{\op},I)$, with these respective coefficient actions.
\end{thm}

Section~\ref{sec-complements} studies complements and supplements of ideals. For a Hall ideal $I$ of a finite skew brace $E$, Theorem~\ref{main7} gives compatible additive and multiplicative complements and proves the existence of a sub-skew brace complement under several hypotheses, including $I\leq\Ker(\lambda)$, supersolubility of $E$, or nilpotency of $(E,+)$. After the first version of this paper, the existence of complements for Hall ideals in arbitrary finite skew has been proved  (\cite[Theorem~A]{Damele26},  \cite[Theorem~2.2]{FT26}). The same result  also  follows from \cite[Theorem~2]{BPP26}, together with classical Schur--Zassenhaus theorem in group theory. The existence of Sylow $p$-sub-skew braces in every finite skew brace, for every prime $p$ dividing its order, has also been  established in \cite[Theorem~2.1]{Truman26}, which also   follows from \cite[Theorem~2]{BPP26}. We retain our independent proofs with the hope the methods used might help in investigating structural properties of skew braces; but decided to present them in Appendix A instead of the main text to keep the flow of the paper.   

We present various interesting results on Frattini sub-skew braces, minimal and split extensions of skew left braces. For finite skew braces $E_1,E_2$, we prove
\[
\Phi(E_1\oplus E_2)\leq\Phi(E_1)\oplus\Phi(E_2),
\]
with equality when both $\Phi(E_1)$ and $\Phi(E_2)$ are ideals (Theorem~\ref{frattini-direct-sum}). A simple skew brace $Q$ of order $24$ with $|\Phi(Q)|=2$ gives strict inclusion for $Q\oplus Q$ (Remark~\ref{simple-nonzero-frattini}).

Section~\ref{sec-reductions} gives criteria for splitting extensions by direct products of ideals. Its main reduction is the following.
\begin{thm}\label{intro-reduction}
	 Let $E$ be a skew brace with an ideal $I$ that is the internal direct product $I=I_1\times\cdots\times I_r$ of ideals $I_j$ of $E$. Here $r\geq1$ and $I_k^* = \prod_{j \neq k} I_j$, with empty sums and products of ideals interpreted as $0$. Then the following  hold true:
\begin{itemize}
	\item[(i)] $E$ splits over $I$ if and only if $E/I_k^*$ splits over $I/I_k^*$ for all $k = 1, \ldots, r$.
	\item[(ii)] If $r\geq2$, then $E$ splits over $I$ if and only if, for any distinct $k,m\in\{1,\ldots,r\}$, the quotients $E/I_k$ and $E/I_m$ split over $I/I_k$ and $I/I_m$, respectively.
\end{itemize}
\end{thm}

We apply this criterion when $I$ is finite and both $(I,+)$ and $(I,\circ)$ are nilpotent.

We use $0$ for the common identity of the two groups of a skew brace and for the one-element skew brace. The additive and multiplicative inverses of $a$ are denoted by $-a$ and $a^{-1}$, respectively. We write
\[
[a,b]^+=-a-b+a+b,\qquad
[a,b]^\circ=a^{-1}\circ b^{-1}\circ a\circ b
\]
for the two group commutators; $[a,b]^+$ is also denoted by $\gamma_+(a,b)$. Higher commutators are defined iteratively. The notation $B\leq A$ means that $B$ is a sub-skew brace of $A$, and $B\trianglelefteq A$ means that $B$ is an ideal. Where the ideal $I$ is understood, we write $\bar A=A/I$. We use the isomorphism theorems for skew braces \cite{BEP24} without further reference, and denote the category of skew left braces by $\mathcal{SB}$.

\section{Preliminaries}\label{sec-preliminaries}

Throughout, a \emph{skew brace} means a skew left brace, and a \emph{brace} means a left brace, that is, a skew brace with abelian additive group. We also say that such a skew brace is of \emph{abelian type}. A \emph{right skew brace} satisfies the analogous compatibility condition with distributivity from the right, and a \emph{two-sided skew brace} is both a left and a right skew brace. A skew brace is \emph{trivial} if its two operations coincide. Thus a trivial skew brace is a group with the same operation used twice, and a trivial brace is an abelian group. A skew brace is \emph{almost trivial} if $a\circ b=b+a$ for all $a,b\in A$. We write $\Aut(A,+)$ for the automorphism group of $(A,+)$.

For two elements $a, b \in A$,  $-b + a + b$ denotes the conjugate of $a$ by $b$ in $(A, +)$ and    $ b^{-1} \circ a \circ b$ denotes  the conjugate of $a$ by $b$ in $(A, \circ)$. 

A \emph{sub-skew brace} of a skew brace $A$ is a subset that is a subgroup of both $(A,+)$ and $(A,\circ)$. A \emph{left ideal} of $A$ is a subgroup $I\leq(A,+)$ invariant under every $\lambda_a$, $a\in A$; it is then also a subgroup of $(A,\circ)$. A left ideal is \emph{strong} if it is normal in $(A,+)$, and is an \emph{ideal} if it is normal in both groups. The kernel $\Ker(\lambda)$ is a sub-skew brace; when $(A,+)$ is abelian, it equals $\Soc(A)$ and is an ideal. A non-zero skew brace is \emph{simple} if its only ideals are $0$ and the whole skew brace.

The \emph{socle} and \emph{annihilator} of a skew brace $A$ are the ideals
\[
\Soc(A)=\Ker(\lambda)\cap\Z(A,+),\qquad
\Ann(A)=\Soc(A)\cap\Z(A,\circ).
\]

The following result is well known.
\begin{prop}
	Let $A$ be a skew brace and let $S$ and $I$ be, respectively, a sub-skew brace and an ideal of $A$. Then $S \circ I = S + I$ is a sub-skew brace of $A$. Moreover, if $S$ is  a left ideal (an ideal), then $S + I$ is also  a left ideal (an ideal).
\end{prop}

	Let $A$ be a skew brace and $B$ a sub-skew brace of $A$. We say that $B$ is {\it supplemented} in $A$ if  there exists a sub-skew brace $C$ such that
$$A = B \circ C =  B + C.$$
Such a $C$ is called a supplement of $B$ in $A$.
The sub-skew brace  $B$ is said to be  \emph{complemented} in $A$ if $B$ is supplemented by  $C$ and 
$B \cap C = \{0\}$. In this case we say that $C$ is a complement of $B$ in $A$.
Note that if $A =  B + C$, where $B$  is a sub-skew brace of $A$ and $C$ is an ideal of $A$, then $A = C + B = B \circ C = C \circ B$. Indeed, if $a=b+c$ with $b\in B$ and $c\in C$, then $a=b\circ\lambda_b^{-1}(c)$. Thus, an ideal $I$ of $A$ is complemented by a skew brace $B$ in $A$ if and only if $A = I + B$ and $I \cap B =  \{0\}$.

An \emph{extension} of a skew brace $H$ by a skew brace $I$ is an exact sequence
\[
0\longrightarrow I\xrightarrow{i}E\xrightarrow{\pi}H\longrightarrow0
\]
of skew brace homomorphisms. We identify $I$ with the ideal $i(I)=\ker\pi$ when the injection is understood. A \emph{set-theoretic section}, abbreviated to \emph{st-section}, is a map $s\colon H\to E$ with $s(0)=0$ and $\pi s=\Id_H$. A section that is a skew brace homomorphism is called a \emph{splitting section}; the extension is \emph{split} if it admits one. We also refer to $E$ itself as the extension.

Let $0\to I\to E\xrightarrow{\pi}H\to0$ be an extension with $I$ a trivial brace, and choose an st-section $s$. The induced actions on $(I,+)$ are
\begin{align*}
\nu_h(x)&=\lambda_{s(h)}(x),\\
\mu_h(x)&=-s(h)+x+s(h),\\
\sigma_h(x)&=s(h)^{-1}\circ x\circ s(h).
\end{align*}
Here $\nu\colon(H,\circ)\to\Aut(I,+)$ is a homomorphism, while $\mu\colon(H,+)\to\Aut(I,+)$ and $\sigma\colon(H,\circ)\to\Aut(I,+)$ are antihomomorphisms. Products of automorphisms denote composition, with the rightmost map applied first. These actions are independent of the chosen st-section and agree for equivalent extensions \cite{RY24}. The triple of actions $(\nu, \mu, \sigma)$ satisfies the following compatibility conditions:
\[
\mu_{\lambda_{h_1}(h_3)}(y) = \nu_{h_1} ( \mu_{h_3} ( \nu_{h_1}^{-1}(y) )),
\]
\[
\nu_{h_2 + h_3} \left( \sigma_{h_2 + h_3}(y) \right) + \mu_{h_3}(y) = \mu_{h_3} \left( \nu_{h_2} \left( \sigma_{h_2}(y) \right) \right) + \nu_{h_3} \left( \sigma_{h_3}(y) \right),
\]
for all $h_1,h_2,h_3\in H$ and $y\in I$. More generally, a homomorphism $\nu\colon(H,\circ)\to\Aut(I,+)$ and antihomomorphisms $\mu\colon(H,+)\to\Aut(I,+)$ and $\sigma\colon(H,\circ)\to\Aut(I,+)$ form a \emph{good triplet of actions} if they satisfy these two identities. Every good triplet is induced by an extension of $H$ by $I$ \cite{RY24}.

The group $\Lambda_A=(A,+)\rtimes_\lambda(A,\circ)$ has multiplication
\[
(x_1,y_1)(x_2,y_2)
=(x_1+\lambda_{y_1}(x_2),\,y_1\circ y_2).
\]
Its diagonal subgroup is
\[
\Diag(\Lambda_A)=\{(a,a):a\in A\}\cong(A,\circ),
\]
since $(a,a)(b,b)=(a\circ b,a\circ b)$. A skew brace homomorphism $\phi\colon E\to H$ induces the group homomorphism
\[
\phi\times\phi\colon\Lambda_E\longrightarrow\Lambda_H,
\qquad (x,y)\longmapsto(\phi(x),\phi(y)).
\]
This makes $A\mapsto\Lambda_A$ a functor from skew braces to groups. Thus, every brace extension 
\begin{equation}\label{sb-ext}
0 \longrightarrow I \xrightarrow{i} E \xrightarrow{\pi} H \longrightarrow 0,
\end{equation}
induces the following group extension:
\begin{equation}\label{g-ext}
0 \longrightarrow \Lambda_I  \xrightarrow{i \times i} \Lambda_E \xrightarrow{\pi \times \pi} \Lambda_H \longrightarrow 0.
\end{equation}

Two extensions $E$ and $E^{\prime}$ of $H$ by $I$ are \emph{equivalent} if there is a skew brace isomorphism $\phi\colon E\to E^{\prime}$ making the following diagram commute:
$$\begin{CD}
 0 @>>> I @>i>> E @>{{\pi} }>> H  @>>> 0\\
 &&  @V{\text{Id}}VV @V{\phi}VV @ VV{ \text{Id}}V \\
 0 @>>> I @>i'>> E^\prime @>{{\pi^\prime} }>> H @>>> 0.
\end{CD}$$
The set of all equivalence classes of extensions of $H$ by $I$ is denoted by $\Ext(H, I)$. When we want to emphasize the use of an action $(\nu, \mu, \sigma)$ of $H$ on $I$, then we denote $\Ext(H, I)$ by $\Ext_{(\nu, \mu, \sigma)}(H, I)$.

\begin{prop}
Let \( E \) be a skew brace extension of \( H \) by \( I \). Then the following hold true:
\begin{itemize}
	\item[(i)] If \( E \) is a split extension, then so is  \( \Lambda_E \).
	\item[(ii)] If \( E \) is an  extension of \( H \) by  a trivial brace \( I \), then the extension \( \Lambda_E \) is an abelian extension of \( \Lambda_H \) by \( I \times I \). Furthermore, if  $I \leq \Ann(E)$, then the induced extension of $\Lambda$-groups is a central extension as $\Lambda_{\Ann(E)} \leq \Z(\Lambda_E)$ (see \cite[Propositions~3.3 and~3.4(5)]{RSU24} for details).
	\item[(iii)]  If \( E \) and \( E' \) are equivalent skew brace extensions of \( H \) by \( I \), then \( \Lambda_{E} \) and \( \Lambda_{E'} \) are also equivalent group extensions of \( \Lambda_H \) by \( \Lambda_I \).
\end{itemize}
\end{prop}

Thus the functor $\Lambda$ preserves exact sequences of extensions.

An abelian group $I$ equipped with a good triplet $(\nu,\mu,\sigma)$ of actions of $H$ is called an \emph{$H$-module}. An \emph{$H$-submodule} is a subgroup invariant under all three actions. A homomorphism of $H$-modules is an additive homomorphism commuting with these actions. A module is \emph{simple} if it is non-zero and has no proper non-zero submodules. For an extension $E$ of $H$ by $I$ inducing the given triplet, the $H$-submodules of $I$ are precisely the ideals of $E$ contained in $I$.


\section{From skew braces to groups: extensions and cohomology}\label{sec-cohomology}

Fix a good triplet $(\nu,\mu,\sigma)$ of actions of a skew brace $H$ on a trivial brace $I$. Let $E$ be an extension inducing this triplet and choose an st-section $s\colon H\to E$. The action by conjugation in the induced group extension gives a right action of $\Lambda_H$ on $I\times I$. To compute it, first observe that, for $x\in I$ and $h\in H$,
\begin{align}\label{action1}
	\lambda_x(s(h))-s(h) &= -x+ x \circ s(h)-s(h) \nonumber \\
	&=-x+ s(h) \circ \sigma_h(x)-s(h) \nonumber\\
	&=-x+ s(h)+\nu_h(\sigma_h(x))-s(h) \nonumber\\
	&=-x +\mu_{h}^{-1}(\nu_h(\sigma_h(x))),
\end{align} 
which belongs to $I$. 

For $(h, g) \in \Lambda_H$ and $(x, y) \in I \times I$,  using \eqref{action1} at the fourth equality, we make the following computation  in $\Lambda_E$: 
\begin{align*}
(s(h), s(g))^{-1} (x, y) (s(h), s(g)) &= \big(\lambda_{s(g)^{-1}}(- s(h)), {s(g)^{-1}}\big)(x, y) (s(h), s(g))\nonumber\\
&= \big(\lambda_{{s(g)^{-1}}}(- s(h)) + \lambda_{{s(g)^{-1}}}(x) + \lambda_{{s(g)^{-1}} \circ y} (s(h)), {s(g)^{-1}} \circ y \circ s(g)\big)\nonumber\\
&= \big(\lambda_{s(g)^{-1}}\big(-s(h) + x + \lambda_y(s(h))\big), \sigma_g(y)\big)\nonumber\\
&= \big(\lambda_{s(g)^{-1}}\big(-s(h) + x - y + \mu_h^{-1}(\nu_h(\sigma_h(y))) + s(h)\big), \sigma_g(y)\big)\nonumber\\
&=  \big(\nu^{-1}_g\big(\mu_h(x - y) + \nu_h(\sigma_h(y))\big),  \sigma_g(y)\big)\nonumber\\
&=  \big(\nu^{-1}_g(\mu_h(x - y)) + \nu^{-1}_g(\nu_h(\sigma_h(y))),  \sigma_g(y)\big).
\end{align*}
Thus the induced right action $\chi\colon\Lambda_H\to\Aut(I\times I)$ is given by
\begin{equation}\label{Lambda-action}
\chi_{(h,g)}(x,y) = \big(\nu^{-1}_g(\mu_h(x - y)) + \nu^{-1}_g(\nu_h(\sigma_h(y))),  \sigma_g(y)\big).
\end{equation}
The formula depends only on the good triplet, not on $E$ or $s$. We use the right-action convention $\chi_{uv}=\chi_v\chi_u$ in the group cohomology calculations below. If $I\leq\Ker(\lambda)$, then the action simplifies to
\begin{equation}\label{ker-action}
\chi_{(h,g)}(x,y) = \big(\nu_g^{-1}(\mu_h(x)), \sigma_g(y)\big).
\end{equation}
Indeed, if $\lambda_x(w) = w$ for all $x \in I$ and $w \in E$, then, for all $h \in H$ and $y \in I$, we get 
\begin{align*}
	\nu_h (\sigma_h(y))& =\nu_h(s(h)^{-1} \circ y \circ s(h) )= -s(h)+ (y \circ s(h)) = -s(h)+ y + \lambda_y(s(h)) \\
&=  -s(h)+ y +s(h) =\mu_h(y).
\end{align*}

The good-triplet identities are exactly the conditions under which \eqref{Lambda-action} defines a right action.

\begin{thm}\label{rightactionthm}
	Let $H$ be a skew brace and $I$ an abelian group. Let
	\[
	\mu: (H, +) \rightarrow \Aut(I) \quad \text{and} \quad \sigma: (H, \circ) \rightarrow \Aut(I)
	\]
	be antihomomorphisms, and let $\nu: (H, \circ) \rightarrow \Aut(I)$ be a homomorphism. Define a map
	\[
	\chi^{(\nu, \mu, \sigma)} : \Lambda_H \longrightarrow \Aut(I \times I)
	\]
	by setting
	\[
	\chi^{(\nu, \mu, \sigma)}_{(h,g)}(x,y) = \left(  \nu^{-1}_g(\mu_h(x - y)) + \nu^{-1}_g(\nu_h(\sigma_h(y))),  \sigma_g(y)\right)
	\]
	for all $h, g \in H$ and $x, y \in I$. Then $\chi^{(\nu, \mu, \sigma)}$ is an anti-homomorphism if and only if $(\nu, \mu, \sigma)$ is a good triplet of actions.
\end{thm}
\begin{proof}
Suppose that $\chi^{(\nu, \mu, \sigma)}$ is an anti-homomorphism. Then it is easy to see that the subgroup $I \times \{0\}$ is invariant under the action of  $\chi^{(\nu, \mu, \sigma)}$. 
Indeed, for any $h, g \in H$ and $x \in I$, we have
\[
\chi^{(\nu, \mu, \sigma)}_{(h,g)}(x,0) = \left( \nu^{-1}_g(\mu_h(x)) , \; 0\right).
\]
Note that $(0,g)(h,0) = (\lambda_g(h), g)$. Using this and the fact that $\chi^{(\nu, \mu, \sigma)}$ is an anti-homomorphism, we obtain
\begin{align*}
	\left( \mu_h(\nu^{-1}_g(x)), \; 0 \right)  &=  \chi^{(\nu, \mu, \sigma)}_{(h,\;0)}
	 \chi^{(\nu, \mu, \sigma)}_{(0,\;g)}(x,\;0)  = 	\chi^{(\nu, \mu, \sigma)}_{(0,\;g)(h,\;0)}(x,\;0)\\
	&=   \chi^{(\nu, \mu, \sigma)}_{(\lambda_g(h),\; g)}(x, \;0) = \bigl(\nu^{-1}_g(\mu_{\lambda_g(h)}(x)),0\bigr),
	\end{align*}
which is the first condition of the good triplet of actions $(\nu, \mu, \sigma)$.

Next, we verify the second compatibility condition. Since $\chi^{(\nu, \mu, \sigma)}$ is an anti-homomorphism, for all $h, g \in H$ and $y \in I$, we have
\begin{align*}
	\chi^{(\nu, \mu, \sigma)}_{(h,0)} \chi^{(\nu, \mu, \sigma)}_{(g,0)}(0,y) 
	= \chi^{(\nu, \mu, \sigma)}_{(g+h,0)}(0,y).
\end{align*}

Making further computations, we obtain
\begin{align} \label{second}
	\chi^{(\nu, \mu, \sigma)}_{(h,0)} \chi^{(\nu, \mu, \sigma)}_{(g,0)}(0,y) 
	&= \chi^{(\nu, \mu, \sigma)}_{(h,0)} \left( -\mu_g(y) + \nu_g(\sigma_g(y)), \; y \right) \notag \\
	&= \left( \mu_h\!\left( -\mu_g(y) + \nu_g(\sigma_g(y)) - y \right) + \nu_h(\sigma_h(y)), \; y \right)
\end{align}
and
\begin{align}\label{third}
	\chi^{(\nu, \mu, \sigma)}_{(g+h,0)}(0,y) 
	&= \left( -\mu_{g+h}(y) + \nu_{g+h}(\sigma_{g+h}(y)), \; y \right).
\end{align}
Comparing the first coordinates, we deduce
\begin{align*}
	\nu_{g+h}(\sigma_{g+h}(y)) 
	= \mu_h\!\big(\nu_g(\sigma_g(y))\big) - \mu_h(y) + \nu_h(\sigma_h(y)),
\end{align*}
which is the second condition of the  good triplet of actions $(\nu, \mu, \sigma)$.

Conversely, suppose that $(\nu, \mu, \sigma)$ is a good triplet of actions. The extension construction in \cite{RY24} gives an extension $E$ of $H$ by $I$ inducing $(\nu,\mu,\sigma)$. This induces a group extension $\Lambda_E$ of $\Lambda_H$ by $I \times I$. By \eqref{Lambda-action}, the map $\chi^{(\nu, \mu, \sigma)}$ describes the corresponding action of $\Lambda_H$ on $I \times I$. Thus, $\chi^{(\nu, \mu, \sigma)}$ is an anti-homomorphism.
\end{proof}

\begin{cor}\label{inducedaction}
	Let $H$ be a skew brace and let $I$ be an $H$-module via a good triplet of actions $(\nu, \mu, \sigma)$ on $I$. Then we can define an action of $\Lambda_H$ on $I$ by
	\[
	\phi^{(\nu, \mu)}_{(h, g)} = \mu^{-1}_{h} \nu_{g}, \quad \text{for all } h,g \in H.
	\]
\end{cor}
\begin{proof}
By Theorem \ref{rightactionthm}, we have
\[
\chi_{(h,g)}(x,0)=(\nu_g^{-1}\mu_h(x),0).
\]
Thus, $I\times\{0\}$ is $\chi$-invariant. Moreover,
$\psi_{(h,g)}:=\nu_g^{-1}\mu_h$ is the restriction of the right action
$\chi$ to $I\times\{0\}$. Hence, the map
$(h,g)\mapsto\psi_{(h,g)}$ is an anti-homomorphism. Consequently, taking pointwise inverses gives the homomorphism
\[
\phi^{(\nu,\mu)}_{(h,g)}
=\psi_{(h,g)}^{-1}
=\mu_h^{-1}\nu_g.
\]
The proof is complete.
 \end{proof}

\begin{remark}\label{remark-ideal}
Let $G$ be a skew brace and $I$ be its ideal. Then equation \eqref{action1} also holds in more generality, that is, $\lambda_x(g)-g \in I$ for all $x \in I$ and $g \in G$. The proof is essentially the same and uses the fact that $I$ is normal in both the groups $(G, +)$ and $(G, \circ)$.
\end{remark}

One might expect that a non-split brace extension $0 \to I \to E \to H \to 0$ always goes to a non-split extension $0 \to \Lambda_I \to \Lambda_E \to \Lambda_H \to  0$. But this is not the case, as explained in the following example.  
\begin{example}\label{exp0}
Let $A$ be the brace with $(A, +) = \mathbb{Z}_2 \times \mathbb{Z}_4$, whose socle has order $2$, and $(A, \circ)$ defined as follows:
\[
(x_1, y_1) \circ (x_2, y_2)
\;:=\;
\bigl(x_1 + x_2,\; y_1 + y_2 + 2y_1x_2 + 2(x_1 + y_1)y_2 \bigr).
\]
In this case, as classified in \cite{Bach15}, we get
\[
(A, \circ) \;\cong\; (\mathbb{Z}/2)^3.
\]
It is easy to see that 
\[
I = \mathbb{Z}_2 \times \langle \bar{2} \rangle \;\cong\; \mathbb{Z}_2 \times \mathbb{Z}_2
\]
is an ideal of $A$, which is a trivial brace. The GAP id of $A$ is \texttt{SmallSkewbrace(8,8)}. A computation by GAP \cite{GAP}, using the package YangBaxter \cite{VK24}, shows that the ideal $I$ does not admit a complement in $A$, but $\Lambda_I$ does admit a complement in $\Lambda_A$. Hence, the extension $0 \to I \to A \to A/I \to 0$ of  skew braces does not split, but the group extension $0 \to \Lambda_I \to \Lambda_A \to \Lambda_A/\Lambda_I \to 0$  does split.
\end{example}
 
The converse of the statement that the functor $\Lambda$ preserves splitting therefore requires additional hypotheses. Let $E$ be an extension of $H$ by a trivial brace $I$ inducing $(\nu,\mu,\sigma)$. We call it a \emph{socle extension} if $I\leq\Soc(E)$. This is equivalent to $\mu_h=\Id$ and $\nu_h^{-1}=\sigma_h$ for all $h\in H$. Indeed,
if $I \le \Soc(E)$, then $\mu$ is trivial, and using  \eqref{action1}, that is,
\[
0 = \lambda_x(s(h)) - s(h) = -x + \mu_{h}^{-1}\bigl(\nu_h(\sigma_h(x))\bigr),
\]  
we obtain $\nu_h(\sigma_h(x)) = x$ for all $x \in I$.  
For the converse, if $\mu$ is trivial, then $I \subseteq \Z(E,+)$.  
To show that $I \subseteq \ker \lambda$, it is enough to check that $\lambda_x(s(h)) = s(h)$ for all $h \in H$ and $x \in I$, which easily follows from \eqref{action1}.  For these extensions, splitting is equivalent to splitting of the associated group extension.

\begin{thm}\label{prop-split-ext}
	The skew brace extension \eqref{sb-ext}  with $I \le \Soc(E)$ splits if and only if  the associated group extension \eqref{g-ext} splits.
\end{thm}

\begin{proof}
Note that `only if' part of the assertion follows from the construction.  So assume that the group extension \eqref{g-ext} splits. Then there exists a section (group homomorphism) 
$s : \Lambda_H \to \Lambda_E$ such that  \( (\pi \times \pi) \circ s = \mathrm{Id}_{\Lambda_H} \). For all \( h, g \in H \), we can write
$$
s(h, g) = (s_1(h, g), s_2(h, g)),$$
where \( s_1, s_2 : H \times H \to E \) are the component maps of $s$.
Since \( s \) is a group homomorphism, for all \( (h_1, g_1), (h_2, g_2) \in \Lambda_H \), we have
$$s(h_1 + \lambda_{g_1}(h_2),\; g_1 \circ g_2) = s(h_1, g_1) \cdot s(h_2, g_2).$$
Now writing in the  component form and making further computations we get
$$s(h_1 + \lambda_{g_1}(h_2),\; g_1 \circ g_2) = (s_1(h_1 + \lambda_{g_1}(h_2),\; g_1 \circ g_2),\; s_2(h_1 + \lambda_{g_1}(h_2),\; g_1 \circ g_2)) $$
and 
\begin{align*}
s(h_1, g_1) \cdot s(h_2, g_2)  &= (s_1(h_1, g_1), s_2(h_1, g_1)) \cdot (s_1(h_2, g_2), s_2(h_2, g_2)) \\
		&= \left(s_1(h_1, g_1) + \lambda_{s_2(h_1, g_1)}(s_1(h_2, g_2)),\; s_2(h_1, g_1) \circ s_2(h_2, g_2)\right).
	\end{align*}
	Comparing the first components, we obtain the key relation:
	\begin{equation}\label{eq:main1}
		s_1(h_1 + \lambda_{g_1}(h_2),\; g_1 \circ g_2) = s_1(h_1, g_1) + \lambda_{s_2(h_1, g_1)}(s_1(h_2, g_2)).
	\end{equation}
	
	Now, $s$ being a section, we see that \( \pi(s_1(h, g)) = h \) and \( \pi(s_2(h, g)) = g \), and hence the difference \( -s_1(g,h)+s_2(h,g)\in I \). So there exists a  map \( f : H \times H \to I \) such that 
		\[
	s_2(h, g) = s_1(g, h) + f(h, g).
	\]
 In particular, $s_1(0, h)$ and $s_2(h, 0)$ both lie in $I$ for all $h \in H$. Observe that for any $h_1, h_2 \in H$, we have
	\[
	s(h_1, h_2) = s((h_1, 0)(0, h_2)),
	\]
	which, using \eqref{eq:main1},  implies that
	\begin{equation}\label{eq:addsplit}
		s_1(h_1, h_2) = s_1(h_1, 0) + s_1(0, h_2),
	\end{equation}
	since $s_2(h_1, 0) \in I$.

Using the fact that $I \le \Soc(E)$ and the identity $s_2(h, g) = s_1(g, h) + f(h, g)$, \eqref{eq:main1} takes the form
	\begin{equation}\label{eq:main2}
		s_1(h_1 + \lambda_{g_1}(h_2),\; g_1 \circ g_2) = s_1(h_1, g_1) + \lambda_{s_1(g_1, h_1)}(s_1(h_2, g_2)).
	\end{equation}
Define a map \( t : H \to E \) by
	\[
	t(h) := s_1(h, 0).
	\]
	Since $s(0,0)=(0,0)$ and $\pi(s_1(h,0))=h$, the map $t$ is an st-section of $\pi$. We claim that \( t \) is a skew brace homomorphism. Taking  \( g_1 = g_2 = 0 \) in  \eqref{eq:main2}, we get
	$$s_1(h_1 + h_2, 0) = s_1(h_1, 0) + \lambda_{s_1(0, h_1)}(s_1(h_2, 0))$$
	for all  \( h_1, h_2 \in H \). 	Since \( s_1(0, h_1) \in I \), this, in the form of $t$, gives
	\[
	t(h_1 + h_2) = t(h_1) + t(h_2).
	\]
This shows that \( t \) is a homomorphism of the additive groups. Next, taking  \( h_1 = 0 \) and \( g_2 = 0 \) in \eqref{eq:main2}, we get
$$s_1(\lambda_{g_1}(h_2), g_1) = s_1(0, g_1) + \lambda_{s_1(g_1, 0)}(s_1(h_2, 0))$$
for all $g_1, h_2 \in H$, which,
	using  \eqref{eq:addsplit} on the left hand side and the fact $I \le \Soc(E)$, gives
	\[
	s_1(\lambda_{g_1}(h_2), 0) = \lambda_{s_1(g_1, 0)}(s_1(h_2, 0)).
	\]
	This, using the definition of $t$, gives 
	\[
	t(\lambda_{g_1}(h_2)) = \lambda_{t(g_1)}(t(h_2)).
	\]
Since $t$ is  a homomorphism of the additive groups, we get $t(g_1 \circ h_2) = t(g_1) \circ t(h_2)$ for all $g_1, h_2 \in H$. Hence $t$ is a skew brace homomorphism, and the proof is complete. 		
\end{proof}

We now compute the group cocycle associated with a brace extension. Throughout, a map $\alpha\colon H\times H\to I$ is \emph{normalized} if $\alpha(0,h)=\alpha(h,0)=0$ for all $h\in H$. Let \( E \) be a brace extension of   \( H \) by  \( I \),  and let \( s: H \rightarrow E \) be an st-section. For \( h_1, h_2 \in H \), the associated brace $2$-cocycle $(\beta,\tau)$ is defined by
\[
\beta(h_1, h_2) := -s(h_1 + h_2) + s(h_1) + s(h_2),
\]
\[
\tau(h_1, h_2) := -s(h_1 \circ h_2) +  (s(h_1) \circ s(h_2)) = \nu_{h_1 \circ h_2}(\overline{\tau}(h_1, h_2)),
\]
where the multiplicative group $2$-cocycle is
\[
\overline{\tau}(h_1, h_2) := s(h_1 \circ h_2)^{-1} \circ s(h_1) \circ s(h_2).
\]
Define $f: H \times H \rightarrow I$ by $f(h_1, h_2)=-s(\lambda_{h_1}(h_2))+ \lambda_{s(h_1)}(s(h_2))$. We have 
\begin{align}
-s(\lambda_{h_1}(h_2))+ \lambda_{s(h_1)}(s(h_2))=& -s(-h_1 + (h_1 \circ h_2)) - s(h_1) + (s( h_1) \circ s(h_2))\nonumber\\
=& \beta(-h_1, h_1 \circ h_2)-s(h_1 \circ h_2) -s(-h_1) -s(h_1) + (s( h_1) \circ s(h_2))\nonumber\\
=&  \beta(-h_1, h_1 \circ h_2)-s(h_1 \circ h_2) -\beta(h_1, -h_1) + (s( h_1) \circ s(h_2)) \nonumber\\
=&   \beta(-h_1, h_1 \circ h_2) -\mu_{h_1 \circ h_2}(\beta(h_1, -h_1)) -s(h_1 \circ h_2)+ (s( h_1) \circ s(h_2)) \nonumber\\
=&   \beta(-h_1, h_1 \circ h_2) -\mu_{h_1 \circ h_2}(\beta(h_1, -h_1)) + \tau(h_1, h_2 ). \label{f-map}
\end{align}

Now we compute  $2$-cocycles corresponding to the extension $\Lambda_E$ of $\Lambda_H$ by $\Lambda_I$, with respect to the st-section $S:=s \times s$. For  $h_1, h_2, g_1, g_2 \in H$, using \eqref{f-map}, we have
\begin{align*}
 &S((h_1,g_1) (h_2, g_2))^{-1} S(h_1, g_1) S(h_2, g_2)\\
&= \big((s \times s)((h_1+ \lambda_{g_1}(h_2), g_1 \circ g_2)) \big)^{-1}(s(h_1), s(g_1) ) (s(h_2), s(g_2))\\
&= \big((s(h_1+ \lambda_{g_1}(h_2)), s(g_1 \circ g_2)) \big)^{-1} \big( s(h_1)+ \lambda_{s(g_1)}(s(h_2)), s(g_1) \circ s(g_2) \big)\\
&= \big( -\lambda^{-1}_{s(g_1 \circ g_2)}\big(s(h_1+ \lambda_{g_1}(h_2))\big), s(g_1 \circ g_2)^{-1} \big) \big( s(h_1)+ \lambda_{s(g_1)}(s(h_2)), s(g_1) \circ s(g_2) \big)\\
&=  \big( -\lambda^{-1}_{s(g_1 \circ g_2)}\big(s(h_1+ \lambda_{g_1}(h_2))\big) + \lambda^{-1}_{ s(g_1 \circ g_2)}\big(s(h_1)+ \lambda_{s(g_1)}(s(h_2))\big),  s(g_1 \circ g_2)^{-1} \circ s(g_1) \circ s(g_2) \big)\\
&=  \big( \lambda^{-1}_{s(g_1 \circ g_2)}\big(-s(h_1+ \lambda_{g_1}(h_2))+ s(h_1)+ \lambda_{s(g_1)}(s(h_2)) \big), \bar{\tau}(g_1, g_2)\big)\\
&=   \big( \lambda^{-1}_{s(g_1 \circ g_2)}\big(-s(h_1+ \lambda_{g_1}(h_2))+ s(h_1)+ s(\lambda_{g_1}(h_2))-s(\lambda_{g_1}(h_2))+  \lambda_{s(g_1)}(s(h_2))\big), \bar{\tau}(g_1, g_2)\big)\\
&=  \big( \lambda^{-1}_{s(g_1 \circ g_2)}\big(\beta(h_1, \lambda_{g_1}(h_2))+ f(g_1, h_2)\big),  \bar{\tau}(g_1, g_2)\big)\\
&= \big(\nu^{-1}_{g_1 \circ g_2} \big(\beta(h_1, \lambda_{g_1}(h_2))+ f(g_1, h_2)\big), \bar{\tau}(g_1, g_2) \big).
\end{align*}

Let $\Z^2_{Sb}(H,I)$ denote the abelian group of normalized brace
$2$-cocycles with coefficients in $I$ with respect to the good
triplet of actions $(\nu,\mu,\sigma)$. Also, let
$\Z^2_{Gp}((H,+),I)$ and
$\Z^2_{Gp}((H,\circ),I)$ denote the sets of group
$2$-cocycles of the groups $(H,+)$ and $(H,\circ)$, respectively,
with coefficients in $I$ with respect to the actions $\mu$ and
$\sigma$, respectively. For a map $\tau:H\times H\to I$, define
\[
\overline{\tau}(h_1,h_2)
:=\nu_{h_1\circ h_2}^{-1}(\tau(h_1,h_2)).
\]
Then, by \cite[Theorem~3.3]{RY24},
$(\beta,\tau)\in\Z^2_{Sb}(H,I)$ if and only if
\[
\beta\in\Z^2_{Gp}((H,+),I)
\quad\text{and}\quad
\overline{\tau}\in\Z^2_{Gp}((H,\circ),I),
\]
and the following equivalent form of the compatibility condition holds:
\begin{align}
	&\mu_{\lambda_{h_1}(h_3)}(\tau(h_1,h_2))
	-\tau(h_1,h_2+h_3)+\tau(h_1,h_3) \nonumber\\
	&\qquad=\nu_{h_1}(\beta(h_2,h_3))
	+\mu_{h_1\circ h_3}(\beta(h_1,-h_1))
	-\beta(-h_1,h_1\circ h_3) \nonumber\\
	&\qquad\quad-\beta(h_1\circ h_2,\lambda_{h_1}(h_3)).
	\label{mutual1}
\end{align}

In particular, the group $2$-cocycle conditions for $\beta$ and
$\overline{\tau}$ are given by
\begin{align}
	\beta(h_1+h_2,h_3)
	+\mu_{h_3}\bigl(\beta(h_1,h_2)\bigr)
	&=\beta(h_1,h_2+h_3)+\beta(h_2,h_3),
	\tag{C$_\beta$}\\
	\overline{\tau}(h_1\circ h_2,h_3)
	+\sigma_{h_3}\bigl(\overline{\tau}(h_1,h_2)\bigr)
	&=\overline{\tau}(h_1,h_2\circ h_3)
	+\overline{\tau}(h_2,h_3).
	\tag{C$_\tau$}
\end{align}

\begin{lemma}\label{Lemma 3.6A}
Let $E$ be a brace extension of $H$ by an abelian group $I$, and let $s$ be a
st-section.	The following statements hold:
	\begin{enumerate}[(a)]
		\item If $(\beta,\tau)$ is the cocycle of $E$ with respect to
		$s$, then, for all $g,h\in H$, we have
		\[
		-s(\lambda_g(h))+\lambda_{s(g)}(s(h))
		=\tau(g,h)-\beta\bigl(g,\lambda_g(h)\bigr).
		\]
		
		\item If $\beta$ is any normalized map satisfying
		$(\mathrm{C}_\beta)$, then
		\[
		\beta(-g,g\circ h)-\mu_{g\circ h}(\beta(g,-g))
		=-\beta(g,\lambda_g(h)).
		\]
		Hence, the map $f$ in \eqref{f-map} is given by
		\[
		f(g,h)=\tau(g,h)-\beta\bigl(g,\lambda_g(h)\bigr).
		\]
	\end{enumerate}
\end{lemma}

\begin{proof}
	\begin{enumerate}[(a)]
\item We know that
\begin{align*}
	\lambda_{s(g)}(s(h))
	&=-s(g)+s(g)\circ s(h)\\
	&=-s(g)+s(g\circ h)+\tau(g,h).
\end{align*}
Also,
\[
s(g\circ h)
=s(g+\lambda_g(h))
=s(g)+s(\lambda_g(h))-\beta(g,\lambda_g(h)).
\]
Substituting this in the above expression, we get
\begin{align*}
	\lambda_{s(g)}(s(h))
	&=-s(g)+s(g)+s(\lambda_g(h))
	-\beta(g,\lambda_g(h))+\tau(g,h).
\end{align*}
Hence,
\[
\tau(g,h)-\beta(g,\lambda_g(h))
=-s(\lambda_g(h))+\lambda_{s(g)}(s(h)).
\]
This proves the result.

\item Taking $(h_1,h_2,h_3)=(g,-g,g\circ h)$ in
$(\mathrm{C}_\beta)$ and using $-g+g\circ h=\lambda_g(h)$
and the normalization $\beta(0,g\circ h)=0$, we obtain
\[
\mu_{g\circ h}(\beta(g,-g))
=\beta(g,\lambda_g(h))+\beta(-g,g\circ h).
\]
\end{enumerate}
The proof is complete.	
\end{proof}

The preceding calculation gives the following characterization of brace cocycles in terms of group cocycles.

\begin{thm}\label{Theorem 3.7}
	Let a skew brace $H$ act on a trivial brace $I$ via a good triplet
	$(\nu,\mu,\sigma)$, and let
	$\beta,\tau\colon H\times H\to I$ be normalized maps.
	Define $c=c_{(\beta,\tau)}\colon\Lambda_H\times\Lambda_H\to I\times I$
	by
	\[
	c\bigl((h_1,g_1),(h_2,g_2)\bigr)
	=
	\Bigl(
	\nu^{-1}_{g_1\circ g_2}
	\bigl(
	\beta(h_1,\lambda_{g_1}(h_2))
	-\beta(g_1,\lambda_{g_1}(h_2))
	+\tau(g_1,h_2)
	\bigr),
	\nu^{-1}_{g_1\circ g_2}\bigl(\tau(g_1,g_2)\bigr)
	\Bigr).
	\]
	Then $c$ is a group $2$-cocycle of $\Lambda_H$ with coefficients in
	$I\times I$ with respect to the action $\chi$, that is,
	\begin{equation}\label{3.11prime}
		c(y,z)-c(xy,z)+c(x,yz)-\chi_z\bigl(c(x,y)\bigr)=0
		\qquad (x,y,z\in\Lambda_H),
	\end{equation}
	if and only if $(\beta,\tau)\in\Z^2_{Sb}(H,I)$. Moreover, if $(\beta,\tau)$ is the cocycle of an extension $E$
	with respect to an st-section $s$, then $c_{(\beta,\tau)}$ is
	the cocycle of $\Lambda_E$ with respect to the st-section
	$S=s\times s$.
\end{thm}

\begin{proof}
	
If $(\beta,\tau)\in\Z^2_{Sb}(H,I)$, then we know that there exists a skew brace extension $E$ of $H$ by $I$ and an st-section $s$ of $E$ such that the $2$-cocycle corresponding to $s$ is $(\beta, \tau)$. Hence, it follows from the preceding  construction that the map $c$ is a group $2$-cocycle of $\Lambda_H$ with coefficients in $I\times I$ with respect to the action $\chi$.

We now prove the `only if' part.
	Suppose that $c=c_{(\beta,\tau)}$ is a group $2$-cocycle of
	$\Lam H$ with coefficients in $(I\times I,\chi)$. We show that
	$(\beta,\tau)\in\ZSb(H,I)$.
	
	Since $\beta$ and $\tau$ are normalized, we have
	\[
	c\bigl((h_1,0),(h_2,0)\bigr)
	=(\beta(h_1,h_2),0)
	\]
	and
	\[
	c\bigl((0,g_1),(0,g_2)\bigr)
	=(0,\overline{\tau}(g_1,g_2)),
	\]
	where
	\[
	\overline{\tau}(g_1,g_2)
	=\nu^{-1}_{g_1\circ g_2}(\tau(g_1,g_2)).
	\]
	Moreover, we have
	\[
	\chi_{(h,0)}(u,0)=(\mu_h(u),0)
	\quad\text{and}\quad
	\chi_{(0,g)}(0,v)=(0,\sigma_g(v)).
	\]
	Thus, evaluating the cocycle identity \eqref{3.11prime} on triples
	$(h_1,0),(h_2,0),(h_3,0)$ gives $(\mathrm{C}_\beta)$.
	Similarly, evaluating it on triples
	$(0,g_1),(0,g_2),(0,g_3)$ gives $(\mathrm{C}_\tau)$.
	
	It remains to verify the compatibility condition
	\eqref{mutual1}. For this, take
	\[
	x=(0,g),\qquad y=(h_2,0),\qquad z=(h_3,0).
	\]
	Then
	\[
	xy=(\lambda_g(h_2),g)
	\quad\text{and}\quad
	yz=(h_2+h_3,0).
	\]
	Substituting these elements in \eqref{3.11prime}, and comparing the
	first coordinates, we obtain, after applying $\nu_g$ and
	using the good triplet identity
	\[
	\nu_g\mu_{h_3}\nu_g^{-1}
	=\mu_{\lambda_g(h_3)},
	\]
	the following equality:
	\begin{eqnarray*}
		&& \beta(\lambda_g(h_2),\lambda_g(h_3))
		-\beta(g,\lambda_g(h_3))
		+\tau(g,h_3) +\mu_{\lambda_g(h_3)}
		\bigl(\tau(g,h_2)-\beta(g,\lambda_g(h_2))\bigr)\\
		 & & \qquad= \tau(g,h_2+h_3)
		-\beta(g,\lambda_g(h_2)+\lambda_g(h_3))
		+\nu_g(\beta(h_2,h_3)).
	\end{eqnarray*}
	
	Now, applying $(\mathrm{C}_\beta)$ with
	\[
	(h_1,h_2,h_3)
	=(g,\lambda_g(h_2),\lambda_g(h_3)),
	\]
	and using $g+\lambda_g(h_2)=g\circ h_2$, we get
	\begin{align*}
		\beta(g,\lambda_g(h_2)+\lambda_g(h_3))
		={}&\beta(g\circ h_2,\lambda_g(h_3)) +\mu_{\lambda_g(h_3)}
		\bigl(\beta(g,\lambda_g(h_2))\bigr)\\
		&-\beta(\lambda_g(h_2),\lambda_g(h_3)).
	\end{align*}
	Substituting this expression in the preceding equality and
	cancelling the common terms, we obtain
	\begin{align}
		&\mu_{\lambda_g(h_3)}(\tau(g,h_2))
		-\tau(g,h_2+h_3)+\tau(g,h_3)\nonumber\\
		&\qquad=
		\nu_g(\beta(h_2,h_3))
		+\beta(g,\lambda_g(h_3))
		-\beta(g\circ h_2,\lambda_g(h_3)).
		\label{310prime}
	\end{align}
	
	Finally, by Lemma~\ref{Lemma 3.6A}(b), we have
	\[
	\beta(g,\lambda_g(h_3))
	=\mu_{g\circ h_3}(\beta(g,-g))
	-\beta(-g,g\circ h_3).
	\]
	Using this identity in \eqref{310prime}, we obtain
	\begin{align*}
		&\mu_{\lambda_g(h_3)}(\tau(g,h_2))
		-\tau(g,h_2+h_3)+\tau(g,h_3)\\
		&\quad=
		\nu_g(\beta(h_2,h_3))
		+\mu_{g\circ h_3}(\beta(g,-g))
		-\beta(-g,g\circ h_3) -\beta(g\circ h_2,\lambda_g(h_3)).
	\end{align*}
	This is precisely the compatibility condition
	\eqref{mutual1}. Hence, by \cite[Theorem~3.3]{RY24},
	the conditions $(\mathrm{C}_\beta)$, $(\mathrm{C}_\tau)$ and
	\eqref{mutual1} imply that $(\beta,\tau)\in\ZSb(H,I)$.
	This completes the proof.
\end{proof}

We remark that the map
\[
(\beta,\tau)\longmapsto c_{(\beta,\tau)}
\]
from $\ZSb(H,I)$ to
$\Z^2_{Gp}(\Lambda_H,I\times I)$ is additive and injective.
Indeed, the additivity follows directly from the definition of
$c_{(\beta,\tau)}$. To prove injectivity, suppose that
$c_{(\beta,\tau)}=0$. Then its second component gives
\[
\nu^{-1}_{g_1\circ g_2}\bigl(\tau(g_1,g_2)\bigr)=0
\]
for all $g_1,g_2\in H$, and hence $\tau=0$.
Furthermore, by taking $g_1=g_2=0$, we obtain
\[
c_{(\beta,\tau)}\big((h_1,0),(h_2,0)\big)
=(\beta(h_1,h_2),0).
\]
Thus $\beta=0$, and consequently $(\beta,\tau)=0$.
Therefore, the map is injective.

We now prove that this map induces a group homomorphism at the cohomology level. Example~\ref{exp0} shows that the induced homomorphism need not be injective. The following corollary gives a sufficient condition for injectivity.

\begin{thm}\label{Psi-hom}
Let $H$ be a skew brace acting on a trivial brace $I$ via the actions $(\nu, \mu, \sigma)$. Then the  map $\Psi: \Ho^2_{Sb}(H, I) \to \Ho^2_{Gp}(\Lambda_H, I \times I)$, defined by
$$\Psi([(\beta, \tau)]) = [c_{(\beta,\tau)}]$$
is a group homomorphism.
\end{thm}
\begin{proof}
Equivalent brace extensions induce equivalent group extensions. Hence the map
\[
\Ext_{(\nu,\mu,\sigma)}(H,I)\longrightarrow
\Ext_{\chi}(\Lambda_H,I\times I),\qquad
[E]\longmapsto[\Lambda_E],
\]
is well defined. Under the extension--cohomology correspondence \cite{RY24}, the cocycle calculation above identifies this map with $\Psi$. The map $\Psi$ is additive because the map $(\beta,\tau)\mapsto c_{(\beta,\tau)}$ is additive.

\end{proof}

For socle extensions, Theorem~\ref{prop-split-ext} gives the following injectivity criterion.
\begin{cor}\label{useful-cor}
Let $H$ be a skew brace acting on an abelian group $I$ via a good triplet of actions $(\nu, \mu, \sigma)$ with $\nu_h^{-1} = \sigma_h$ and $\mu_h = \Id$ for all $h \in H$. Then the map
	\[
	\Psi: \Ho^2_{Sb}(H, I) \rightarrow \Ho^2_{Gp}(\Lambda_H, I \times I)
	\]
	of Theorem \ref{Psi-hom} is an embedding.
\end{cor}

Taking $I=\mathbb C^*$ with trivial actions recovers \cite[Theorem~2.2]{LV24}. In the multiplier calculations below, all coefficient actions on $\mathbb C^*$ are trivial.

 \begin{prop}
The exponent of the Schur multiplier of a finite skew brace $H$ divides $|H|^2$.
\end{prop}
\begin{proof}
Recall  that  if $G$ is a group and $I_1$ and $I_2$ are $G$-modules, then
$$\Ho^2_{Gp}(G, I_1 \oplus I_2) \cong \Ho^2_{Gp}(G, I_1) \oplus \Ho^2_{Gp}(G, I_2).$$
By the preceding result, by taking $I = \mathbb{C}^*$, the group $\Ho^2_{Sb}(H, \mathbb{C}^*)$, the Schur multiplier of the skew brace $H$, embeds into
$$\Ho^2_{Gp}(\Lambda_H, \mathbb{C}^* \times \mathbb{C}^*) \cong \Ho^2_{Gp}(\Lambda_H, \mathbb{C}^*) \oplus \Ho^2_{Gp}(\Lambda_H, \mathbb{C}^*).$$
It is a known group theory result that the exponent of $\Ho^2_{Gp}(\Lambda_H, \mathbb{C}^*)$ divides $|\Lambda_H|=|H|^2$. Hence the exponent of $\Ho^2_{Sb}(H, \mathbb{C}^*)$  divides $|H|^2$, which completes the proof.
\end{proof}

We write $M_b(H)=\Ho^2_{Sb}(H,\mathbb C^*)$ and let $M(G)=H_2(G,\mathbb Z)$ be the group Schur multiplier. For finite $G$, the universal coefficient theorem gives
\[
\Ho^2_{Gp}(G,\mathbb C^*)\cong\Hom(M(G),\mathbb C^*)\cong M(G).
\]
In the group multiplier decompositions below, induced actions on $M(G)$ are those on homology. We now give further applications of Corollary~\ref{useful-cor}.

\begin{prop}
	Let $H$ be a finite trivial or almost trivial skew brace. Then the exponent of the Schur multiplier of the skew brace $H$ divides the order of $H$.
\end{prop}
\begin{proof}
Put $G=(H,\circ)$. If $H$ is a trivial or almost trivial skew brace, then $\Lambda_H\cong G\times G$. It is well known that
\[
\Ho^2_{Gp}(G\times G,\mathbb{C}^*)
\cong \Ho^2_{Gp}(G,\mathbb{C}^*)\oplus\Ho^2_{Gp}(G,\mathbb{C}^*)
\oplus\bigl(G/\gamma_2(G)\otimes G/\gamma_2(G)\bigr)
\]
(see \cite{GK87}). The exponent of each Schur multiplier factor divides $|G|$, and the exponent of the tensor factor divides $\exp(G/\gamma_2(G))$, hence also $|G|$. The assertion follows from Corollary~\ref{useful-cor}.
\end{proof}

The preceding result also follows from the explicit formula in \cite[Proposition~2.4]{LV24}. Here a group $G$ is viewed as a trivial skew brace, and $G^{\mathrm{op}}$ denotes its opposite skew brace.

\begin{thm}
	Let $G$ be a finite group. Then $M_b(G) \cong M_b(G^{\mathrm{op}}) \cong M(G) \times (G/\gamma_2(G) \otimes G/\gamma_2(G))$.
\end{thm}

Another application to braces of prime power orders is the following.
\begin{prop}
	Let $(H, +, \circ)$ be a finite skew brace of abelian type having order $p^n$, where $p$ is an odd prime. Then the exponent of $M_b(H)$ divides $|H|$.
\end{prop}

\begin{proof}
	Put $N=(H,+)$ and $Q=(H,\circ)$. Then $\Lambda_H=N\rtimes_\lambda Q$ is a finite $p$-group with $N$ abelian and $p$ odd. Thus \cite[Theorem 2.1]{Evens72} gives
	\[
	M(\Lambda_H) \cong M(H, \circ) \oplus H_1((H, \circ), (H, +)) \oplus M(H, +)_{(H, \circ)},
	\]
	where $M(H, +)_{(H, \circ)}$ is the maximal $(H, \circ)$-trivial quotient of $M(H, +)$, defined as
	\[
	M(H, +)_{(H, \circ)} = M(H, +) / \langle h \cdot x - x \mid h \in H, x \in M(H, +) \rangle^+,
	\]
	with `$\cdot$' denoting the action of $(H, \circ)$ on $M(H, +)$. 
	
	The exponent of $M(Q)$ divides $|Q|=|H|$. The group $H_1(Q,N)$ is annihilated by $\exp(N)$, and $M(N)_Q$ is a quotient of $M(N)$, whose exponent divides $|N|=|H|$. Thus each factor has exponent dividing $|H|$. Corollary~\ref{useful-cor} now gives the assertion.
\end{proof}

The multiplicative cocycle also gives an additive map
\[
\Ho^2_{Sb}(H,I)\longrightarrow\Ho^2_{Gp}((H,\circ),I),
\qquad [(\beta,\tau)]\longmapsto[\overline\tau],
\]
where the target uses the right action $\sigma$. The group cocycle condition for $\overline\tau$ was established above. To check that the map is well defined, it remains to compare coboundaries. Assume that $(\beta_1, \tau_1)$ and $(\beta_2, \tau_2)$ are cohomologous. Then there exists a map $\theta: H \to I$ such that (see  \cite[proof of Theorem 3.6]{RY24})
$$(\tau_1 - \tau_2)(h_1, h_2) = \nu_{h_1}(\theta(h_2)) - \theta(h_1 \circ h_2) + \nu_{h_1 \circ h_2}(\sigma_{h_2}(\nu^{-1}_{h_1}(\theta(h_1)))).$$
For $h_1, h_2 \in H$, we now compute
\begin{align*}
(\bar{\tau}_1 - \bar{\tau}_2)(h_1, h_2) &= \nu^{-1}_{h_1 \circ h_2} \left( (\tau_1 - \tau_2)(h_1, h_2) \right)\\
	&= \nu^{-1}_{h_1 \circ h_2} \big( \nu_{h_1}(\theta(h_2)) - \theta(h_1 \circ h_2) + \nu_{h_1 \circ h_2}(\sigma_{h_2}(\nu^{-1}_{h_1}(\theta(h_1)))) \big)\\
	&= \theta_1(h_2) - \theta_1(h_1 \circ h_2) + \sigma_{h_2}(\theta_1(h_1)),
\end{align*}
where $\theta_1(h) = \nu^{-1}_h(\theta(h))$. Hence, $\bar{\tau}_1$ and $\bar{\tau}_2$ are cohomologous via the map $\theta_1$.

The diagonal subgroup defined in Section~\ref{sec-preliminaries} acts on $I\times I$ by restricting $\chi$. Since
\[
\chi_{(g,g)}(y,y)=(\sigma_g(y),\sigma_g(y)),
\]
this action preserves $\Diag(\Lambda_I)$. It therefore induces an action on $J=\Lambda_I/\Diag(\Lambda_I)$ given by
	\[
	\chi_{(g,g)}(\overline{(x, y)}) = \overline{(\nu_g^{-1} \mu_g(x - y), 0)}.
	\]
	This action of $\Diag(\Lambda_H)$ is equivalent to the action of $\Diag(\Lambda_H)$ on $I \times \{0\}$; that is, $I \times \{0\}$ is isomorphic to $J$ as a $\Diag(\Lambda_H)$-module under the map
	\[
	(x, 0) \mapsto \overline{(x, 0)}.
	\]
The isomorphism $\Diag(\Lambda_H)\cong(H,\circ)$, given by $(g,g)\mapsto g$, identifies this action with the right action
\[
\psi_g=\nu_g^{-1}\mu_g\qquad(g\in H).
\]
Write $I_\psi$ for $I$ with this action. The map $\overline{(x,y)}\mapsto x-y$ identifies $J$ with $I_\psi$, so \cite[Proposition 2.2]{PSY18} gives
\[
\Ho^2_{Gp}\big(\Diag(\Lambda_H),J\big)
\cong\Ho^2_{Gp}\big((H,\circ),I_\psi\big).
\]
The action $\psi$ equals $\sigma$ if and only if $\mu_g=\nu_g\sigma_g$ for every $g\in H$, equivalently $I\leq\Ker(\lambda)$ in the extension, by \eqref{action1}.

We conclude with the behaviour of extensions and cohomology under passage to the opposite skew brace.
 For a skew brace $E := (E, +, \circ)$, we define the opposite skew brace  $E^{\op} := (E, +^{\op}, \circ)$. First let $I$ be an arbitrary skew brace. If $E$ is an extension of $H$ by $I$, then $E^{\op}$ is an extension of $H^{\op}$ by $I^{\op}$, as $(E/I)^{\op} \cong E^{\op}/I^{\op}$ under the map $x + I \mapsto x +^{\op} I^{\op}$ for all $x \in E$. Let $\Ext(X, I)$ denote the set of equivalence classes of all skew brace extensions of a skew brace $X$ by $I$. Let $\varphi: \Ext(H, I) \rightarrow \Ext(H^{\op}, I^{\op})$ be a map given by   $[E] \mapsto [E^{\op}]$. 

\begin{prop}\label{bijection1}
The map $\varphi$ is a well-defined bijective map.
\end{prop}

Now let $I$ be a trivial brace such that a skew brace $H$ acts on it via  $(\nu, \mu, \sigma)$, and $E$ be an extension of $H$ by $I$
\[
0 \longrightarrow I \xrightarrow{i} E \xrightarrow{\pi} H \longrightarrow 0.
\]
Let 
\[
0 \longrightarrow I \xrightarrow{i} E^{\op} \xrightarrow{\pi} H^{\op} \longrightarrow 0,
\]
be the corresponding opposite extension. We determine the action associated with $E^{\op}$.
Let $s$ be an st-section of $E$. Then $s$ is also an st-section of $E^{\op}$. For $h \in H^{\op}$ and $y \in I$, we have
$$\mu^{\op}_h(y) = -s(h) +^{\op} y +^{\op} s(h) = s(h) + y - s(h) = \mu_{-h}(y).$$
Similarly, we compute
\begin{align*}
	\nu^{\op}_h(y) 
	&= -s(h) +^{\op} (s(h) \circ y) = (s(h) \circ y) - s(h) \\
	&= s(h) - s(h) + (s(h) \circ y) - s(h) = \mu_{-h}(\nu_h(y)).
\end{align*}
Moreover, $\sigma^{\op}_h = \sigma_h$. Hence, this shows that the map $\varphi$ defined in Proposition~\ref{bijection1} restricts to a bijection
\[
\varphi : \Ext_{(\nu, \mu, \sigma)}(H, I) \longrightarrow \Ext_{(\nu^{\op}, \mu^{\op}, \sigma^{\op})}(H^{\op}, I).
\]


Next we explore the connection between second cohomology groups of $H$ and $H^{\op}$ with coefficients in a trivial brace $I$. For $h_1, h_2 \in H$, we have
\begin{align*}
\beta^{\op}(h_1, h_2) & = -s(h_1+^{\op}h_2)+^{\op} s(h_1) +^{\op}s(h_2)\\
& =s(h_2)+s(h_1) -s(h_2+h_1)\\
& =s(h_2)+s(h_1) -s(h_2+h_1)+s(h_2)+s(h_1)-s(h_1)-s(h_2)\\
&= \mu_{-(h_2+h_1)}(\beta(h_2, h_1))
\end{align*}
and 
\begin{align*}
\tau^{\op}(h_1, h_2) & =-s(h_1 \circ h_2)+^{\op} (s(h_1) \circ s(h_2))\\
&=(s(h_1) \circ s(h_2)) - s(h_1 \circ h_2)\\
&= (s(h_1) \circ s(h_2)) -s(h_1 \circ h_2) + (s(h_1) \circ s(h_2)) - (s(h_1) \circ s(h_2))\\
&= \mu_{-(h_1 \circ h_2)}(\tau(h_1, h_2)), \mbox{ (since }s(h_1) \circ s(h_2) - s(h_1 \circ h_2) \in I).
\end{align*}
These formulas give the following result, which generalizes \cite[Theorem 1.37]{LV24}.
\begin{thm}
Let $H$ be a skew brace acting on a trivial brace $I$ via a good triplet $(\nu,\mu,\sigma)$. Compute the cohomology of $H^{\op}$ with the induced triplet $(\nu^{\op},\mu^{\op},\sigma^{\op})$ above. Then the following hold:
\begin{itemize}
\item[(i)] The pair $(\beta, \tau) \in \Z^2_{Sb}(H,I)$ if and only if  $(\beta^{\op}, \tau^{\op}) \in \Z^2_{Sb}(H^{\op},I)$.
\item[(ii)] The map $\tilde{\varphi}: \Ho^2_{Sb}(H, I) \rightarrow \Ho^2_{Sb}(H^{\op}, I)$ given by $[(\beta, \tau)] \mapsto [(\beta^{\op}, \tau^{\op})]$
is an isomorphism.
\end{itemize}
\end{thm}


\section{Complements and supplements}\label{sec-complements}

We first study complements of Hall ideals and then turn to reduced products, minimal extensions and Frattini sub-skew braces. The group-theoretic counterparts of the minimal-extension results appear in \cite{Hill72}. 

\subsection{Complements in skew braces}
An ideal $I$ of a finite skew brace $E$ is said to be  a \emph{Hall ideal} if the order of $I$ is coprime to the order of the quotient skew brace  $E/I$. 
\begin{prop}
Let $0 \longrightarrow I \xrightarrow{i} E \xrightarrow{\pi} H \longrightarrow 0$ be a brace extension of a skew brace $H$ by a skew brace $I$. Then the brace extension $E$  splits if and only if there exists a sub-skew brace $G$ of $E$ such that $G+I = E$ and $G \cap I = \{0\}$, that is, $I$ admits a complement $G$ in $E$. 
\end{prop}
\begin{proof}
First assume that the brace extension $E$ splits. Then there exists a skew brace homomorphism $s : H \to E$ such that $\pi s = \Id_H$. Then obviously $s(H)$ is the required $G$.
Conversely, note  that the map $\pi|_G: G \rightarrow H$ is an isomorphism of skew braces. Define $s: H \rightarrow E$ by $s(h) = (\pi|_G)^{-1}(h)$. Then $s$ is a skew brace homomorphism and  a section of $\pi$. Hence $E$ is a split extension. 
\end{proof}

In view of this result, while dealing with natural extension $0 \to I \to E \to G \to 0$, where $I$ is an ideal of the skew brace $E$ and $G = E/I$, we usually say that {\it $E$ splits over $I$} whenever $I$ admits a complement in $E$. 

A finite skew brace $E$ is said to be \emph{soluble} if it  admits an abelian series of ideals, that is, a series of ideals of $E$
	\[
	E = I_n \supseteq I_{n-1} \supseteq \cdots \supseteq I_0 = \{0\},
	\]
such that $I_{i}/I_{i-1}$ is a trivial skew brace of abelian type for every $1 \leq i \leq n$.
The  skew brace $E$ is said to be \emph{supersoluble} if it is soluble with 
each quotient  $I_{i+1}/I_i$  trivial skew brace of prime order.

Let $\pi$ be a set of primes and let $\pi'$ be the set of primes outside $\pi$. A positive integer is a $\pi$-number if all its prime divisors belong to $\pi$, and a $\pi'$-number if all its prime divisors belong to $\pi'$. A sub-skew brace $K$ of a finite skew brace $E$ is a Hall $\pi$-sub-skew brace if $|K|$ is a $\pi$-number and $[E:K]=|E|/|K|$ is a $\pi^{\prime}$-number. If $\pi=\{p\}$, such a $K$ is called a Sylow $p$-sub-skew brace of $E$. 
We first state the following result  \cite[Theorem~19]{CDMFT25}:

\begin{thm}\label{CDMFT}
A finite supersoluble skew brace contains Hall $\pi$-sub-skew braces for every set $\pi$ of primes.
\end{thm}

Our next result gives conditions under which a Hall ideal admits a complement and describes properties of the additive and multiplicative complements. As stated in the introduction, this result has now been proved in full generality; therefore to keep the flow, we present our independent proof in  Appendix A.

\begin{thm}\label{main7}
	Let $E$ be an extension of $H$ by $I$, where $H$ and $I$ are finite skew braces of coprime orders.  
	Let $|H| = m$ and $|I| = n$. Then the following statements hold:
	
	\begin{enumerate}
		\item There exist complements $H_1$ of $(I,+)$ in $(E,+)$ and $H_2$ of $(I,\circ)$ in $(E,\circ)$ such that 
		$\lambda_g(h) \in H_1$ for all $g \in H_2$ and $h \in H_1$. Moreover
		\[
		\lambda_x(h) \equiv h \pmod{I}
		\]
	for all  $x \in I$ and $h \in H_1$. 
	
	\item  For any $z \in E$ and $h \in H_1$, there exists an element $g \in H_2$ such that 
		\[
		\lambda_{z}(h) \equiv \lambda_{g}(h) \pmod{I}.
		\]
		
	\item Suppose that either
	\begin{itemize}
		\item[(i)] $\lambda_x(h)$ commutes with $h$ in $(E,+)$ for all $x\in I$ and $h\in H_1$, or
		\item[(ii)] $I\leq\Ker(\lambda)$.
	\end{itemize}
	Then $\lambda_x(h)=h$ for all $x\in I$ and $h\in H_1$, and $H_1$ is a left ideal and a complement of $I$ in $E$.

	\item For each $h\in H_1$, write uniquely
        \[
        h=g_h\circ\phi(h),\qquad g_h\in H_2,\quad\phi(h)\in I.
        \]
        This defines a map $\phi\colon H_1\to I$. Then $H_1$ is a sub-skew brace of $E$ if and only if $\lambda_{\phi(h)}(h_1) = h_1$ for all $h, h_1 \in H_1$. Thus,  if $\lambda_{\phi(h)}(h_1) = h_1$ for all $h, h_1 \in H_1$, then $I$ admits a complement in $E$.
	
	\item Let $h\in H_1$ and write $h=g_h\circ\phi(h)$ with $g_h\in H_2$ and $\phi(h)\in I$. Then the order of $h$ divides the order of $g_h$ in $(E,+)$, and the order of $g_h$ divides the order of $h$ in $(E,\circ)$.

\item If $E$ is supersoluble, then $I$ admits a complement in $E$.

\item If $H_1$ is the only complement of $I$ in $(E, +)$, then $H_1$ is a left ideal of $E$, and therefore a complement of $I$. In particular, $I$ admits a complement if $(E, +)$ is nilpotent.

\end{enumerate}
\end{thm}

\begin{prop}
Let $I$ be an ideal of a skew brace $E$ such that $I$ is a trivial brace and $I\leq\Ker(\lambda)$. If $G_1,G_2$ are complements of $I$ in $E$, then there exists a skew brace automorphism of $E$ fixing $I$ pointwise and mapping $G_1$ to $G_2$. 
\end{prop}
\begin{proof}
By the given hypothesis, it follows that  for any element $a \in E$, there exist unique elements $u \in I$ and $g_1 \in G_1$ depending on $a$  such that $a = u + g_1 = u \circ g_1$. Now, for  $g_1 \in G_1$, there exist unique elements $x \in I$ and $g_2 \in G_2$ depending on $g_1$  such that 
$$g_1 = x + g_2 = x \circ g_2.$$ 
By the uniqueness of factorization, it follows that $g_2 = -x + g_1 = x^{-1} \circ g_1$.  Define a map $f: E \to E$ by setting
$$f(a) := u + g_2,$$
where $g_1 = x + g_2$. Then obviously $f$ is well defined, $f(u) =u$ for all $u \in I$ and $f(g_1) = f(0 + g_1) = g_2$. Similarly, we can define $f': E \to E$ by exchanging the roles of $G_1$ and $G_2$, and can prove that $f f' = f' f = \Id_E$. It only remains to prove that $f$ and $f'$ are brace homomorphisms. We only prove it for $f$. Let $a, a' \in E$ with unique factorizations  $a = u + g_1$, $a' = u' + g'_1$, where $u, u' \in I$ and $g_1, g'_1 \in G_1$. Then $a + a' = u + g_1 + u' - g_1 + g_1 + g'_1$. Also $g_1$ and $g'_1$ admit unique factorizations $g_1 = x + g_2$ and $g'_1 = x' + g'_2$, where $x, x' \in I$ and $g_2, g'_2 \in G_2$.  Thus 
\begin{eqnarray*}
f(a + a') &=& u + g_1 + u' - g_1 + g_2 + g'_2\\
&=& u + x + (g_2 + u' - g_2) - x   + g_2 + g'_2\\
&=&  u + g_2 + u' + g'_2 \mbox{ (since $I$ is a trivial brace)}\\
&=& f(a) + f(a').
\end{eqnarray*}
Since $I\leq\Ker(\lambda)$, all the factorizations $u+g=u\circ g$ used above coincide. Since $(I,\circ)$ is abelian and normal in $(E,\circ)$, the same calculation with $\circ$ gives $f(a\circ a^{\prime})=f(a)\circ f(a^{\prime})$. The proof is complete.
\end{proof}

We use the following property of minimal non-zero ideals of finite soluble skew braces \cite[Theorem~B]{BEJP24}.

\begin{thm}\label{minimumideal}
Every minimal non-zero ideal of a finite soluble skew brace is a trivial skew brace of abelian type. Moreover, it is an elementary abelian \( p \)-group for some prime $p$.
\end{thm}

We now present a conditional result on the existence of Hall sub-skew braces, whose proof is given in Appendix A.

\begin{thm}\label{hallsoluble}
	Let \( E \) be a finite soluble skew brace  such that $\lambda_x(y)$ commutes with $y$ in $(E,+)$ for all $x,y\in E$.  Then \( E \)  admits a Hall $\pi$-sub-skew brace,  for every subset $\pi$ of the set of all prime divisors of $|E|$.
\end{thm}


As a special case of the preceding theorem, we immediately get
\begin{cor}
	Let \( E \) be a finite soluble skew brace  such that $\lambda_x(y)$ commutes with $y$ in $(E,+)$ for all $x,y\in E$.  Then \( E \)  admits a Sylow $p$-sub-skew brace for every prime divisor $p$ of $|E|$.
\end{cor}

\noindent\textbf{Remark.}  As mentioned in the introduction, the preceding result has now been proved in full generality. We would also record \cite[Theorem~A]{EGGK26}, where it was proved that  every sub-skew brace of prime-power order is contained in a Sylow sub-skew brace for all finite left nilpotent skew braces.

\subsection{Skew braces of square-free order}
Next we consider skew braces of square-free order. We write $I\rtimes_{Sb}H$ for a split extension of $H$ by $I$, with the induced actions understood, and $\C_p$ for the trivial brace on the cyclic group of prime order $p$. 
First, we recall a basic result from  \cite[Corollary 3.9]{BEFPT24}. 
\begin{thm}
	Every finite skew brace of square-free order is supersoluble.
\end{thm}

\begin{prop} \label{solu1}
	Let $E$ be a finite non-zero skew brace of square-free order $n$, and let $n = p_1 p_2 \cdots p_k$ be its prime factorization. 
	Then, after a suitable ordering of the primes,
	\[
	E \cong (\cdots((\C_{p_1} \rtimes_{Sb} \C_{p_2}) \rtimes_{Sb} \C_{p_3}) \rtimes_{Sb} \cdots) \rtimes_{Sb} \C_{p_k}.
	\]
\end{prop}

\begin{proof}
	It follows from the preceding theorem that $E$ is a supersoluble skew brace.  Let
	\[
	E = I_k \supseteq I_{k-1} \supseteq \cdots \supseteq I_0 = \{0\},
	\]
	be a  series of $E$ such that each $I_{i}/I_{i-1}$  is a trivial skew brace of prime order. Without loss of generality, assume  that $I_{i}/I_{i-1} \cong \C_{p_i}$ for all $1 \leq i \le k$.  If $k=1$, the result is immediate. Assume $k\geq2$. Then $I_1\cong\C_{p_1}$ and $I_2/I_1\cong\C_{p_2}$. Since $p_1$ and $p_2$ are coprime, by Theorem~\ref{main7}(6), $I_2$ is a split extension of $\C_{p_2}$ by $\C_{p_1}$. Hence,
	\[
	I_2 \cong \C_{p_1} \rtimes_{Sb} \C_{p_2}.
	\]
	Proceeding inductively, we obtain
	\[
	E \cong (\cdots((\C_{p_1} \rtimes_{Sb} \C_{p_2}) \rtimes_{Sb} \C_{p_3}) \rtimes_{Sb} \cdots) \rtimes_{Sb} \C_{p_k}.
	\]
This completes the proof.
\end{proof}

Next, we recall a description of split extensions of skew braces from \cite[Theorem 3.1, Theorem 3.3]{NR24}.

\begin{thm}\label{splitsbext}
Let $E$ be a split extension of a skew brace $H$ by a skew brace $I$, and choose a splitting section $s\colon H\to E$. Define
\begin{align*}
\nu_h&=\lambda_{s(h)}|_I,\\
\mu_h(y)&=-s(h)+y+s(h),\\
\sigma_h(y)&=s(h)^{-1}\circ y\circ s(h).
\end{align*}
Then $\nu\colon(H,\circ)\to\Aut(I,+)$ is a homomorphism, while $\mu\colon(H,+)\to\Aut(I,+)$ and $\sigma\colon(H,\circ)\to\Aut(I,\circ)$ are antihomomorphisms. The bijection $(h,y)\mapsto s(h)+y$ identifies $E$ with $I\rtimes_{(\nu,\mu,\sigma)}H$, with operations
	\begin{align}
		(h_1, y_1) + (h_2, y_2) &= (h_1 + h_2,\, \mu_{h_2}(y_1) + y_2), \label{sb+} \\
		(h_1, y_1) \circ (h_2, y_2) &= \big(h_1 \circ h_2,\, 
		\nu_{h_1 \circ h_2}\big(\sigma_{h_2}(\nu_{h_1}^{-1}(y_1)) \circ \nu_{h_2}^{-1}(y_2)\big)\big). \label{sbcirc}
	\end{align}
\end{thm}

We now provide an application of Proposition \ref{solu1}.
\begin{cor}
	Let $E$ be a finite skew brace of square-free order $n$ with prime factorization  $n = p_1 p_2 \cdots p_k$ and $p_i \nmid (p_j - 1)$ for all $1 \leq i, j \leq k$. 
	Then $E$ is a trivial skew brace with cyclic additive group.
\end{cor}

\begin{proof}
	If $k\leq1$, the assertion is immediate, so assume $k\geq2$. Using Proposition~\ref{solu1}, we have 
	\[
	E \cong (\cdots((\C_{p_1} \rtimes_{Sb} \C_{p_2}) \rtimes_{Sb} \C_{p_3}) \rtimes_{Sb} \cdots) \rtimes_{Sb} \C_{p_k}.
	\]
	Since $p_2 \nmid (p_1 - 1)$, there cannot exist a nontrivial homomorphism 
	\[
	\C_{p_2} \to \Aut(\C_{p_1}).
	\]
	Hence, the associated actions $\nu$, $\mu$, and $\sigma$ defined in Theorem~\ref{splitsbext} are trivial in this case. Therefore, the skew brace operations defined in \eqref{sb+} and \eqref{sbcirc} coincide, and thus 
	\[
	\C_{p_1} \rtimes_{Sb} \C_{p_2} = \C_{p_1} \times \C_{p_2}
	\]
	is a trivial skew brace of abelian type.  
	
	Next, since $\Aut(\C_{p_1} \times \C_{p_2}) = \Aut(\C_{p_1}) \times \Aut(\C_{p_2})$ (as $p_1$ and $p_2$ are coprime), any nontrivial homomorphism 
	\[
	\C_{p_3} \to \Aut(\C_{p_1}) \times \Aut(\C_{p_2})
	\]
	induces a nontrivial homomorphism 
	\[
	\C_{p_3} \to \Aut(\C_{p_1}) \quad \text{or} \quad \C_{p_3} \to \Aut(\C_{p_2}),
	\]
	which is not possible, since $p_3$  neither divides $p_1 - 1$ nor $p_2 - 1$. Hence, there is no nontrivial homomorphism 
	\[
	\C_{p_3} \to \Aut(\C_{p_1} \times \C_{p_2}).
	\]
	Again using  Theorem~\ref{splitsbext}, we get the  trivial skew brace of abelian type
	\[
	(\C_{p_1} \times \C_{p_2}) \rtimes_{Sb} \C_{p_3} = (\C_{p_1} \times \C_{p_2}) \times \C_{p_3}.
	\]
  Proceeding inductively, we obtain
	\[
	E \cong \C_{p_1} \times \C_{p_2} \times \C_{p_3} \times \cdots \times \C_{p_k},
	\]
	which is a  trivial skew brace  of abelian type. More specifically, $(E, +)$ is a  cyclic group, which completes the proof.
\end{proof}



\subsection{Reduced product of skew braces}
Let $E$ be a skew brace having an ideal $I$, which admits a supplement $G$ in $E$.   Then we call $E$ the {\it reduced product} of $I$ and  $G$ if $G$ admits no proper sub-skew brace which supplements $I$ in $E$. The definition specializes to groups if we consider trivial skew braces.  Let us see an example of a reduced product. 
\begin{example}\label{exp1}
	Let $A$ and $I$  be skew braces  as defined in Example \ref{exp0}. Let $G=\langle(0,1)\rangle^+=\{(0,y):y\in\mathbb{Z}_4\}$. It is easy to see that $G$ is a sub-skew brace of $A$, and we have $I + G = A$. The only non-zero proper sub-skew brace $K$ of $G$ is $  \langle (0,2) \rangle$, and we have $I + K = I$. Thus $A$ is the reduced product of $I$ and $G$.
\end{example}

\begin{lemma}
Let $I$ be an ideal of a skew brace $E$ and suppose that $I$ is a trivial brace. Let $G$ be a sub-skew brace with $G+I=E$. Then $J=I\cap G$ is an ideal of $E$, and $G/J$ is a complement of $I/J$ in $E/J$.
\end{lemma}

\begin{proof}
Since $I$ is an ideal, $J$ is an ideal of $G$. Since $I$ is a trivial brace, $(I,+)=(I,\circ)$ is an abelian group. Thus every element of $I$ commutes with every element of $J$ under both operations. Also, $g\circ x=g+\lambda_g(x)$ and $\lambda_g(I)=I$ give $E=G\circ I=G+I$. Thus $J$ is normal in both groups of $E$. For $g\in G$ and $x\in I$, $\lambda_{g\circ x}|_J=\lambda_g|_J$, since $I$ is a trivial brace. Hence $J$ is $\lambda$-invariant and is an ideal of $E$. In $E/J$, the sub-skew braces $G/J$ and $I/J$ have trivial intersection and their sum is $E/J$, so $G/J$ is a complement of $I/J$.
\end{proof}

 We get the following connection between $E$ and $\Lambda_E$ regarding reduced products.

\begin{prop}
Let $E$ be a skew brace having an ideal $I$ and a sub-skew brace $H$. If $\Lambda_E$ is the reduced  product of $\Lambda_I$ and $\Lambda_H$, then $E$ is  the reduced product of $I$ and $H$.
\end{prop}
\begin{proof}
Let $\Lambda_E = \Lambda_I  \cdot  \Lambda_H$ be the reduced product.  Then for any $(a,0) \in \Lambda_E$,  there exist $x, y \in I$ and $h, g \in H$ such that
$$(a,0) = (x,y)(h,g) = (x + \lambda_y(h), y \circ g) = \left((x + \lambda_y(h) - h) + h, y \circ g\right).$$
As $\lambda_y(h) -h \in I$, this implies $a \in I + H$. So $H$ supplements $I$ in $E$.

Contrarily, assume that  there exists a sub-skew brace $K < H$  of $E$ such that $I + K = E$. Then an easy computation shows that $\Lambda_E = \Lambda_I \cdot \Lambda_K$, which contradicts the minimality of $\Lambda_H$ as $\Lambda_K < \Lambda_H$. Hence $E$ is the reduced product of $I$ and $H$, completing the proof.
\end{proof}

\subsection{Frattini sub-skew braces  and minimal extensions}
Let $E$ be a skew brace. The \emph{Frattini sub-skew brace} of $E$, denoted by $\Phi(E)$, is the intersection of all maximal sub-skew braces of $E$. If there are no maximal sub-skew braces, we set $\Phi(E)=E$.

\begin{prop}\label{prop-reduced}
Let $E = I + H$ be a reduced  product of $I$ and $H$. Then
	\begin{itemize}
		\item[(i)] $I \cap H \leq \Phi(H)$,
		\item[(ii)] $\Phi(E/I) \cong (I + \Phi(H))/I$.
	\end{itemize}
\end{prop}
\begin{proof}
 Put $C=I\cap H$ and suppose that $C\not\leq\Phi(H)$. Then there exists a maximal sub-skew brace \( M \subsetneq H \) such that \( C \not\leq M \). Clearly, as \( I \) is an ideal of \( E \), the set \( C \) is an ideal of \( H \). Using the maximality of \( M \), we have \( C + M = H \). Hence,
\[
E = I + H = I + (C + M) = (I + C) + M = I + M,
\]
which contradicts the minimality of \( H \), and (i) holds.

Now we prove (ii). The map $h+(H\cap I)\mapsto h+I$ is a skew brace isomorphism $H/(H\cap I)\to E/I$. It carries the Frattini sub-skew brace of the source onto that of the target, giving
\[
\Phi(E/I) \cong \Phi(H / (H \cap I)),
\]
where maximal sub-skew braces of $H/(H\cap I)$ are precisely the quotients $M/(H\cap I)$ with $M$ maximal in $H$ and $H\cap I\subseteq M$. But using part (i), \( H \cap I \) is already contained in all maximal sub-skew braces of \( H \). Hence
\[
\Phi(H / (H \cap I)) = \Phi(H) / (H \cap I).
\]
Now it is easy to see that \( H \cap I = \Phi(H) \cap I \), so
\[
\Phi(H) / (H \cap I) = \Phi(H) / (\Phi(H) \cap I) \cong (I + \Phi(H)) / I.
\]
This completes the proof.
\end{proof}

\begin{cor}
Let  $0 \longrightarrow I \xrightarrow{i} E \xrightarrow{\pi} H \longrightarrow 0$
be a brace extension with $E$ finite. Then there exists a sub-skew brace $B$ of $E$ such that $\pi(\Phi(B)) \cong \Phi(H)$.
\end{cor}

\begin{proof}
Since $E$ is finite, choose a supplement $B$ of $I$ of least order. Then $E=I+B$ is a reduced product. By Proposition~\ref{prop-reduced}, applied to this product, the canonical map $B/(B\cap I)\to E/I$ sends $\Phi(B)/(B\cap I)$ onto $\Phi(E/I)$. Composing with the isomorphism $E/I\to H$ induced by $\pi$ gives $\pi(\Phi(B))=\Phi(H)$, as required.
\end{proof}

Let  $0 \longrightarrow I \xrightarrow{i} E \xrightarrow{\pi} H \longrightarrow 0$ be a brace extension. Then $E$ is called a \emph{minimal extension} if $E$ does not admit any proper sub-skew brace $K$ such that $\pi(K) = H$.  Equivalently, $E$ is  said to be minimal with respect to $I$,  if $I$ does not admit a proper supplement in $E$. We first give two examples of minimal extensions.

\begin{example}
	Let $A$ be a skew brace  as in Example \ref{exp0}. It is easy to see that $\Soc(A) = \langle (0,2) \rangle^+$. GAP computations show that there is no proper sub-skew brace $K$ of $A$ such that $K + \Soc(A) = A$. Hence, $A$ is a minimal extension with respect to $\Soc(A)$.
\end{example}	

\begin{example}
Let $A$ be a  brace with additive group $\mathbb{Z}_8$. Define  `$\circ$' on $\mathbb{Z}_8$ as follows:
\[
x \circ y = x + 5^{x}y \pmod{8}.
\]
It turns out that $(A,\circ) \cong \mathbb{Z}_8$ as classified in \cite{Bach15}.  
It is easy to see that $\Soc(A)=\langle 2 \rangle$. Using $\lambda_x(y)=5^x y \pmod{8}$, we note that
\[
\text{if $x$ is odd, then } 5^{x} \equiv 5 \pmod{8}
\mbox{ and }
\text{if $x$ is even, then } 5^{x} \equiv 1 \pmod{8}.
\]
Therefore, every subgroup of $\mathbb{Z}_8$ is an ideal of $A$.  
Since $(A,+)$ is cyclic and $\Soc(A)$ has index two, every proper sub-skew brace is contained in $\Soc(A)$. Hence $A$ is minimal with respect to $\Soc(A)$.
\end{example}

The following consequences concern minimal extensions.
\begin{cor}
A finite brace extension $E$ of a skew brace $H$ by an ideal $I$ is minimal if and only if $I \leq \Phi(E)$.
\end{cor}

\begin{cor}
	Let $E$ be a brace  extension of $H$ by $I$. If $\Lambda_E$ is the minimal extension of $\Lambda_H$ by $\Lambda_I$, then $E$ is also a minimal extension of $H$ by $I$.
\end{cor}

The converse of the above corollary need not be true.  
Let $A:=\texttt{SmallSkewbrace}(4,2)$ and let $I\cong\mathbb{Z}_2$ be its unique ideal of order two. Since $(A,+)\cong\mathbb{Z}_4$, the ideal $I$ has no proper supplement in $A$. Hence $A$ is minimal with respect to $I$. GAP computations give
\[
\Lambda_A\cong\mathbb{Z}_2\times D_4
\qquad\text{and}\qquad
\Lambda_I\cong\mathbb{Z}_2\times\mathbb{Z}_2,
\]
where $D_4$ is the dihedral group of order eight. Since $|\Phi(\Lambda_A)|=2$ and $|\Lambda_I|=4$, the ideal $\Lambda_I$ is not contained in $\Phi(\Lambda_A)$. Thus $\Lambda_A$ is not minimal with respect to $\Lambda_I$.

\begin{prop}
	Let $E$ be a minimal extension of a skew brace $H$ with respect to  $I$ and let $J \leq I$ be such that  $J$ is an ideal of $E$. Then $E/J$ is a minimal extension of $H$ with respect to $I/J$. 
\end{prop}

\begin{proof}
Suppose that \( E/J \) is not a minimal extension of \( H \) by \( I/J \). Then there exists a sub-skew brace \( K \subsetneq E \) containing \( J \) such that \( K/J + I/J = E/J \). Hence, \( E = K + I \), contradicting the minimality of \( E \).
\end{proof}

For extensions by simple coefficient modules, splitting and minimality give the following alternatives. 
\begin{prop}
Let $H$ be a skew brace and $I$ a simple $H$-module. Every extension $E$ of $H$ by $I$ inducing the given actions either splits or is minimal. 
\end{prop}
\begin{proof}
Suppose that \( E \) is not a minimal extension. Then $I$ admits a proper supplement $G$, say, in $E$.  Since $I$ is an ideal of $E$, $G \cap I$ is an ideal of $G$.  We claim that \( G \cap I \) is an ideal of \( E \).  For any  \( h \in E \),  we know that \( h = g \circ z \) for some \( g \in G \), \( z \in I \). Then, for any \( x \in G \cap I \), we get
	\[
	h \circ x \circ h^{-1} = g \circ z \circ x \circ z^{-1} \circ g^{-1} = g \circ x \circ g^{-1}
	\]
belongs to $G$ as well as $I$. Thus \( G \cap I \) is normal in \( (E, \circ) \). Similarly,  \( G \cap I \) is a normal subgroup of \( (E, +) \). Also
	\[
	\lambda_{g \circ z}(x) = \lambda_g(x)
	\]
belongs to $G$ as well as $I$. Hence \( \lambda_h(x) \in G \cap I \), and therefore the claim follows.   Thus, \( G \cap I \) is an \( H \)-submodule of \( I \), which, by the given hypothesis,  is possible only when $G \cap I = \{0\}$ as, $G$ being proper supplement, $G \cap I \ne I$. Hence $G$ is a complement of $I$ in $E$, which  implies that the brace extension \( E \) is a split extension, and the proof is complete.
	\end{proof}

 A concrete case is the following.
\begin{cor}
Let $E$ be an extension of a skew brace $H$ by a cyclic group $I$ of prime order (as a simple skew brace). Then, either $E$ is a split extension or a minimal extension.
\end{cor}

Before proceeding further, we define a generating set of a skew brace. We follow the constructions done in \cite{CFT25, MT23}.  Let $A$ be a skew brace and $X \subseteq  A$.  Denote by $A(X)$ the smallest sub-skew brace of $A$ containing  $X$.  We say that the \emph{ set $X$  generates  $A$}  if  $A = A(X)$ and we write $A = \gen{X}$.  The following closure construction describes the elements of $\gen X$.

Put $K_0=X\cup\{0\}$ and, for $n\geq0$, define
\[
K_{n+1}=K_n\cup\{-a,a^{-1},a+b,a\circ b:a,b\in K_n\},
\qquad
\mathcal W(X)=\bigcup_{n\geq0}K_n.
\]
\begin{prop}\label{helping}
The following statements hold:
\begin{enumerate}
\item[(i)] $\mathcal W(X)=\gen X$.
\item[(ii)] Every element of $\gen X$ is a finite expression in elements of $X\cup\{0\}$ using addition, multiplication and the two inverse operations. A skew brace homomorphism $\phi\colon A\to B$ preserves each such expression. In particular,
$\phi(\gen X)=\gen{\phi(X)}$.
\end{enumerate}
\end{prop}
\begin{proof}
(i) The sets $K_n$ form an increasing sequence containing $0$. If $a,b\in\mathcal W(X)$, some $K_n$ contains both, and $-a,a^{-1},a+b,a\circ b\in K_{n+1}$. Thus $\mathcal W(X)$ is a sub-skew brace containing $X$. Conversely, induction gives $K_n\subseteq\gen X$ for every $n$, proving equality.

(ii) The construction gives the stated finite expressions of every element  $\gen{X}$. Since a skew brace homomorphism preserves both operations, both inverses and $0$, induction gives $\phi(K_n)=K_n(\phi(X))$, where $K_n(\phi(X))$ is constructed in $B$. Taking unions proves the last assertion.
\end{proof}

\begin{thm}\label{generator-lifting}
Let $0 \longrightarrow I \xrightarrow{i} E \xrightarrow{\pi} H \longrightarrow 0$ be a minimal extension of a skew brace $H$ by $I$. If $X\subseteq E$ and $\gen{\pi(X)}=H$, then $\gen X=E$. In particular, if $H$ is finitely generated, then $d(E)=d(H)$, where $d(A)$ denotes the least number of elements generating a finitely generated skew brace $A$.
\end{thm}

\begin{proof}
Set $T := \gen X$. Since $\pi$ is a skew brace homomorphism, $\pi(T)=\gen{\pi(X)}=H$. For each $a\in E$, choose $t\in T$ with $\pi(t)=\pi(a)$. Then $-t+a\in I$, so $a=t+(-t+a)\in T+I$. Thus $T+I=E$, and minimality gives $T=E$. Let $H=\gen{h_1,\ldots,h_r}$ with $r=d(H)$. Since $\pi$ is surjective, we can choose $e_i\in E$ with $\pi(e_i)=h_i$ for each $i$. Applying the first assertion to $X=\{e_1,\ldots,e_r\}$ gives $E=\gen{e_1,\ldots,e_r}$, so $d(E)\leq d(H)$. Conversely, the image  of any generating set of $E$ under $\pi$ generates $H$, so $d(H) \leq d(E)$. The proof is complete.
\end{proof}

We derive the following interesting consequence.

\begin{cor}\label{confrat}
	Let $0 \longrightarrow I \xrightarrow{i} E \xrightarrow{\pi} H \longrightarrow 0$ be an extension of finite skew braces. 
	 If $I\leq\Phi(E)$, then every subset $X\subseteq E$ with $\gen{\pi(X)}=H$ generates $E$. In particular, $d(E)=d(H)$.
\end{cor}

\begin{proof}
 Suppose $I\leq\Phi(E)$ and $K+I=E$ for a sub-skew brace $K$. If $K<E$, choose a maximal sub-skew brace $M$ containing $K$. Then $I\subseteq M$, giving $E=K+I\subseteq M$, a contradiction. Hence $E$ is minimal with respect to $I$, and the assertion follows from Theorem~\ref{generator-lifting}.
 \end{proof}

\noindent\textbf{Remark.} The implication $d(E)=d(H)\Rightarrow I\leq\Phi(E)$ fails in general. For the trivial brace $E=\mathbb{Z}_6$ and $I=\{0,3\}$, the quotient is $H\cong\mathbb{Z}_3$, so $d(E)=d(H)=1$, whereas $\Phi(E)=0$.

We now prove the following fundamental result analogous to Burnside basis theorem in the theory of finite groups. 
\begin{lemma}\label{burnside-basis}
Let $E$ be a skew brace of order $p^n$, where $p$ is prime and $n\geq1$, such that every maximal sub-skew brace of $E$ is an ideal. Set $V := E/\Phi(E)$. Then the following hold.
\begin{enumerate}
\item[(i)] Every maximal sub-skew brace of $E$ has index $p$.
\item[(ii)] $\Phi(E)$ is an ideal, and $V$ is a trivial brace whose additive group is elementary abelian.
\item[(iii)] A subset $X\subseteq E$ generates $E$ if and only if its image spans $V$ over $\mathbb{F}_p$. Consequently, $d(E)=\dim_{\mathbb{F}_p}(V)$.
\item[(iv)] For every ideal $I$ of $E$, every maximal sub-skew brace of $E/I$ is an ideal, and
\[
\Phi(E/I)=(I+\Phi(E))/I,\qquad
d(E/I)=d(E)-\dim_{\mathbb{F}_p}\bigl((I+\Phi(E))/\Phi(E)\bigr).
\]
\item[(v)]  The following are equivalent:
	(a) $I\leq\Phi(E)$,  (b) $d(E)=d(H)$, (c) $E$  is minimal with respect to $I$.
\end{enumerate}
\end{lemma}

\begin{proof}
(i) Let $M$ be a maximal sub-skew brace of $E$. Since $M$ is an ideal, the quotient $Q=E/M$ is a non-zero skew brace of $p$-power order with no non-zero proper sub-skew braces. The fixed-point set
\[
\Fix(Q)=\{a\in Q:\lambda_b(a)=a\text{ for all }b\in Q\}
\]
is an additive subgroup invariant under every $\lambda_b$, hence a left ideal of $Q$. We claim that $Q = \Fix(Q)$. By the given hypothesis, it only requires to be proved that  $|\Fix(Q)| > 1$, which, indeed, follows from the automorphism action of the $p$-group $(Q, \circ)$ on $(Q, +)$ via the $\lambda$ map giving $|\Fix(Q)| \equiv 0  \pmod p$, and  $0 \in \Fix(Q)$.  Thus $Q$ is a trivial skew brace on a $p$-group with no non-zero proper subgroups, giving $|Q|=p$.

(ii) The kernel of the diagonal homomorphism $E\to\prod_M E/M$, with $M$ ranging over the maximal sub-skew braces, is $\Phi(E)$. By (i), each factor is the trivial brace of order $p$. Hence $V$ embeds in a trivial brace on an elementary abelian $p$-group.

(iii) Let $q\colon E\to V$ be the quotient map. Suppose first that $X$ generates $E$. Since $q$ is surjective,
\[
V=q(\gen X)=\gen{q(X)}.
\]
By (ii), the sub-skew braces of $V$ are precisely its $\mathbb{F}_p$-subspaces, so $q(X)$ spans $V$.

Conversely, suppose that $q(X)$ spans $V$, and take $T = \gen X$. Then $q(T)=\gen{q(X)}=V$. For every $a\in E$, there exists $t\in T$ such that $q(a)=q(t)$. Hence $-t+a\in\ker(q) =\Phi(E)$, so $a=t+(-t+a)\in T+\Phi(E)$. Therefore $T+\Phi(E)=E$. If $T<E$, let $M$ be a maximal sub-skew brace of $E$ containing $T$. By definition, $\Phi(E)\subseteq M$, and therefore $E=T+\Phi(E)\subseteq M<E$, a contradiction. Thus $T=E$, so $X$ generates $E$.

Every generating set of $E$ maps to a spanning set of $V$, giving $d(E)\geq\dim_{\mathbb{F}_p}(V)$. Conversely, choose a basis of $V$ and a preimage in $E$ of each basis element. By the preceding argument, these preimages generate $E$, so $d(E)\leq\dim_{\mathbb{F}_p}(V)$.

(iv) If $I=E$, the assertions are immediate. Otherwise, every maximal sub-skew brace of $E/I$ has the form $M/I$, where $M$ is a maximal sub-skew brace of $E$ containing $I$. By the given hypothesis $M$ is an ideal of $E$, so $M/I$ is an ideal of $E/I$.

Write $\bar I=(I+\Phi(E))/\Phi(E)$, and let $q : E \to V$ be the quotient map. Every maximal sub-skew brace $M$ containing $I$ contains $I+\Phi(E)$. Conversely, let $a\notin I+\Phi(E)$, so $q(a)\notin\bar I$. Extend a basis of $\bar I$, together with $q(a)$, to a basis of $V$. Let $W$ be the span of all elements of this basis except $q(a)$. Then $\bar I\subseteq W$, $q(a)\notin W$, and $\dim_{\mathbb{F}_p}(V/W)=1$. Since $V$ is a trivial brace by (b), its two operations coincide, so the vector subspace $W$ is a sub-skew brace of $V$. Thus $M=q^{-1}(W)$ is a sub-skew brace of $E$. Since $\ker(q) = \Phi(E) \subseteq M$, the restriction of $q$ to $M$ gives $M/\Phi(E)\cong W$. Therefore $[E:M]=|E/\Phi(E)|/|M/\Phi(E)|=|V|/|W|=p$, as $V/W$ has dimension one over $\mathbb{F}_p$. A subgroup of prime index admits no intermediate subgroup, so $M$ is a maximal sub-skew brace of $E$. Moreover, $q(I)=\bar I\subseteq W$ gives $I\subseteq M$, and $q(a)\notin W$ gives $a\notin M$. Therefore
\[
\bigcap_{\substack{M\text{ maximal}\\ I\subseteq M}}M=I+\Phi(E).
\]
Taking quotients by $I$ gives $\Phi(E/I)=(I+\Phi(E))/I$. Consequently,
\[
(E/I)/\Phi(E/I)\cong V/\bar I.
\]
Applying (iii) to $E/I$ and using $d(E)=\dim_{\mathbb{F}_p}(V)$ gives
\[
d(E/I)=\dim_{\mathbb{F}_p}(V/\bar I)
=d(E)-\dim_{\mathbb{F}_p} (\bar I).
\]

(v) If $E=0$, the equivalence is immediate. Otherwise, by (iv), $d(H)=d(E)$ if and only if $(I+\Phi(E))/\Phi(E)=0$, equivalently $I  \leq \Phi(E)$. The minimality  of the extension now follows from Corollary \ref{confrat}. Conversely, if $I\not\leq\Phi(E)$, some maximal sub-skew brace $M$ does not contain $I$. The sub-skew brace $M+I$ strictly contains $M$, so $M+I=E$, contradicting the minimality of the extension. The proof is now complete.
\end{proof}

\begin{prop}\label{group-frattini-generators}
Let $0\longrightarrow I\longrightarrow E\xrightarrow{\pi}H\longrightarrow0$ be an extension of finite skew braces. Write $d(A,+)$ and $d(A,\circ)$ for the least numbers of generators of the additive and multiplicative groups of a finite skew brace $A$, and write $\Phi(A,+)$ and $\Phi(A,\circ)$ for their respective group Frattini subgroups.
\begin{enumerate}
\item[(i)] If $I\leq\Phi(E,+)$, then $d(E,+)=d(H,+)$, and $E$ is minimal with respect to $I$.
\item[(ii)] If $I\leq\Phi(E,\circ)$, then $d(E,\circ)=d(H,\circ)$, and $E$ is minimal with respect to $I$.
\end{enumerate}
\end{prop}

\begin{proof}
We only prove assertion (i).   Let $h_1, \ldots, h_r$ generate $(H,+)$, where $r=d(H,+)$. Choose $e_i\in E$ with $\pi(e_i)=h_i$, and let $T$ be the subgroup of $(E,+)$ generated by these elements. Then $\pi(T)=H$, so $T+I=E$. If $T$ is proper, let $M$ be a maximal subgroup of $(E,+)$ containing $T$. Since $I\leq\Phi(E,+)\leq M$, this gives $E=T+I\subseteq M$, a contradiction. Thus $T=E$ and $d(E,+)\leq d(H,+)$. The reverse inequality follows by applying $\pi$ to a generating set of $(E,+)$.

If a proper sub-skew brace $K$ supplements $I$, choose a maximal subgroup $M$ of $(E,+)$ containing $K$. Again $I\subseteq M$, giving $E=K+I\subseteq M$, a contradiction. Hence $E$ is minimal with respect to $I$, completing the proof.
\end{proof}

\noindent\textbf{Remark.} The minimality of an extension does not, by itself, imply equality of the group generator numbers. Let $E$ be the brace on $\mathbb{Z}_4$ with $x\circ y=x+y+2xy$, and let $I=\{0,2\}$. Here $I$ is an ideal and the unique additive subgroup of order two, so it has no proper supplement in $E$. Thus $E$ is minimal with respect to $I$. However, $(E/I,\circ)\cong\mathbb{Z}_2$ and $(E,\circ)\cong\mathbb{Z}_2\times\mathbb{Z}_2$, giving $d(E/I,\circ)=1<2=d(E,\circ)$.

\vspace{1mm}

A skew brace $(A, +, \circ)$ is said to be {\it cyclic} ({\it cocyclic}) if $(A, +)$ ($(A, \circ)$) is a cyclic group.

\begin{prop}\label{cyclic-frattini-structure}
Let $E$ be a finite non-zero skew brace with cyclic additive group of order $n$, and let $r=\operatorname{rad}(n)$ be the product of the distinct prime divisors of $n$. For an integer $m$, write $mE=\{ma:a\in E\}$, where $ma$ is computed in $(E,+)$. Then the following hold.
\begin{enumerate}
\item[(i)] The sub-skew braces of $E$ are exactly the subgroups of $(E,+)$, and all are strong left ideals.
\item[(ii)] $\Phi(E)=rE=\Phi(E,+)$.
\item[(iii)] $E$ has a unique maximal sub-skew brace if and only if $n$ is a prime power. If $n=p^k$, then $pE$ is an ideal, $\Phi(E)=pE$, and every maximal sub-skew brace of $E$ is an ideal.
\item[(iv)] Let $0=I_0<I_1<\cdots<I_k=E$ be a series of sub-skew braces with each index $q_i=[I_{i+1}:I_i]$ prime. Then $I_i\leq\Phi(I_{i+1})$ if and only if $|I_{i+1}|$ is a power of $q_i$, and in that case $I_i=\Phi(I_{i+1})$. Consequently, $I_i=\Phi(I_{i+1})$ for every $0\leq i<k$ if and only if $n$ is a prime power.
\end{enumerate}
\end{prop}

\begin{proof}
(i) Every subgroup of a cyclic group is characteristic, so it is invariant under each additive automorphism $\lambda_a$. It is therefore a left ideal and a sub-skew brace. Since $(E,+)$ is abelian, it is a strong left ideal. Conversely, every sub-skew brace is an additive subgroup.

(ii) It follows  from assertion (i) that the maximal sub-skew braces of $E$ are exactly the maximal additive subgroups, namely $pE$ for the primes $p$ dividing $n$. Their intersection is $rE$, which is also the group Frattini subgroup $\Phi(E,+)$.

(iii) By (i), there is one maximal sub-skew brace for each distinct prime divisor of $n$, proving the equivalence. If $n=p^k$, the unique maximal sub-skew brace is $pE$, so $\Phi(E)=pE$. By (i), $pE$ is a strong left ideal. Moreover, it has index $p$ in the finite $p$-group $(E,\circ)$ and is therefore normal in $(E,\circ)$. Hence $pE$ is an ideal of $E$.

(iv) Set $m := |I_{i+1}|$ and $q := q_i$. The additive group of $I_{i+1}$ is cyclic, so (ii) gives $|\Phi(I_{i+1})|=m/\operatorname{rad}(m)$, while $|I_i|=m/q$. Subgroups of a finite cyclic group are determined by their orders. Hence
$I_i\leq\Phi(I_{i+1})$ if and only if 
$m/q\text{ divides }m/\operatorname{rad}(m)$  if and only if 
$\operatorname{rad}(m)=q$.
This is equivalent to $m$ being a power of $q$. In this case, both $I_i$ and $\Phi(I_{i+1})$ have index $q$ in the cyclic group $(I_{i+1},+)$, so they are equal. If $n$ is a prime power, this applies at every step. Conversely, equality at the final step forces $n$ to be a prime power.
\end{proof}

We first describe the maximal sub-skew braces of a direct sum, whose proof relies on Goursat's lemma (see  \cite{JT97}).

\begin{lemma}\label{maximal-direct-sum}
Let $E_1,E_2$ be finite skew braces. A sub-skew brace $M$ is maximal in $E_1\oplus E_2$ if and only if it has one of the following forms:
\begin{enumerate}
\item[(i)] $M=M_1\oplus E_2$, where $M_1$ is maximal in $E_1$;
\item[(ii)] $M=E_1\oplus M_2$, where $M_2$ is maximal in $E_2$;
\item[(iii)] there are ideals $N_i$ of $E_i$ and a skew brace isomorphism $\theta\colon E_1/N_1\to E_2/N_2$, with these quotients non-zero and simple, such that
\[
M=M_\theta=\{(a,b)\in E_1\oplus E_2:
\theta(a+N_1)=b+N_2\}.
\]
\end{enumerate}
\end{lemma}

\begin{thm}\label{frattini-direct-sum}
Let $E_1,E_2$ be finite skew braces and set $E:= E_1 \oplus E_2$. Then the following hold.
\begin{enumerate}
\item[(i)] $\Phi(E)\leq\Phi(E_1)\oplus\Phi(E_2)$.
\item[(ii)] If $\Phi(E_1)$ and $\Phi(E_2)$ are ideals of $E_1$ and $E_2$, respectively, then
\[
\Phi(E_1\oplus E_2)=\Phi(E_1)\oplus\Phi(E_2).
\]
\end{enumerate}
\end{thm}

\begin{proof}
For (i) and (ii), the assertions are immediate if either factor is zero, so assume both are non-zero.

(i) By Lemma~\ref{maximal-direct-sum}, $M_1\oplus E_2$ is maximal whenever $M_1$ is maximal in $E_1$. Intersecting these sub-skew braces gives $\Phi(E)\subseteq\Phi(E_1)\oplus E_2$. Similarly, $\Phi(E)\subseteq E_1\oplus\Phi(E_2)$, proving (i).

(ii) Suppose that $\Phi(E_i)$ is an ideal of $E_i$ for each $i=1,2$, and put $P=\Phi(E_1)\oplus\Phi(E_2)$. We show that $P$ lies in every maximal sub-skew brace $M$ of $E$. This is immediate for $M$ of types (i) and (ii) in Lemma~\ref{maximal-direct-sum}.

Suppose $M=M_\theta$ is of type (iii), with non-zero simple quotients $Q_i=E_i/N_i$ and quotient maps $q_i\colon E_i\to Q_i$. The preimage under $q_i$ of every maximal sub-skew brace of $Q_i$ is maximal in $E_i$, so $q_i(\Phi(E_i))\subseteq\Phi(Q_i)$. Since $\Phi(E_i)$ is an ideal, its image is an ideal of the simple skew brace $Q_i$. Also, $\Phi(Q_i)<Q_i$, since the finite non-zero skew brace $Q_i$ has a maximal proper sub-skew brace. Hence $q_i(\Phi(E_i))=0$, giving $\Phi(E_i)\subseteq N_i$. Thus every $(a,b)\in P$ satisfies $\theta(a+N_1)=0=b+N_2$, so $P\subseteq M_\theta$. Therefore $P\subseteq\Phi(E)$. Combined with (i), this proves equality.
\end{proof}

\begin{remark}\label{simple-nonzero-frattini}
The inclusion in Theorem~\ref{frattini-direct-sum}(i) can be strict, even when both factors are simple. Let $Q=\texttt{SmallSkewbrace}(24,853)$ in the YangBaxter package. Its additive group is $\mathbb{Z}_6\times\mathbb{Z}_2\times\mathbb{Z}_2$, and its multiplicative group is $S_4$. GAP computations \cite{GAP,VK24,SONATA} show that $Q$ is simple and has exactly two maximal sub-skew braces, of orders six and eight, whose intersection has order two. Thus $\Phi(Q)=\{0,z\}$ with $z\neq0$; in particular, $\Phi(Q)$ is not an ideal.

By Theorem~\ref{frattini-direct-sum}(i) and the maximality of the diagonal $\Delta=\{(a,a):a\in Q\}$ from Lemma~\ref{maximal-direct-sum},
\[
\Phi(Q\oplus Q)\subseteq
\bigl(\Phi(Q)\oplus\Phi(Q)\bigr)\cap\Delta
=\{(0,0),(z,z)\}.
\]
Conversely, $(z,z)$ belongs to every maximal sub-skew brace of $Q\oplus Q$. It belongs to those of types (i) and (ii) in Lemma~\ref{maximal-direct-sum} because $z\in\Phi(Q)$. Since $Q$ is simple, those of type (iii) are graphs of automorphisms of $Q$. Every such automorphism preserves $\Phi(Q)$ and fixes its unique non-zero element $z$, so each graph also contains $(z,z)$. Therefore
\[
\Phi(Q\oplus Q)=\{(0,0),(z,z)\}
\subsetneq\Phi(Q)\oplus\Phi(Q),
\]
the two sides having orders two and four, respectively.
\end{remark}

It is well known that for a finite $p$-group $G$, $\Phi(G)=\gamma_2(G)G^p$, where $G^p$ is the subgroup generated by $\{g^p:g\in G\}$. A description of $\Phi(E)$ for general skew braces of prime-power order would extend the special cases above.

For a finite non-zero cyclic skew brace $E$, Proposition~\ref{cyclic-frattini-structure} gives $\Phi(E)=\Phi(E,+)$, and $E$ has a unique maximal sub-skew brace if and only if $|E|$ is a prime power. For a finite non-zero cocyclic skew brace $E$ of prime-power order, the sub-skew braces form a chain because they are subgroups of the cyclic group $(E,\circ)$. Thus $E$ has a unique maximal proper sub-skew brace, equal to $\Phi(E)$. This sub-skew brace is contained in the unique maximal subgroup of $(E,\circ)$, giving $\Phi(E)\leq\Phi(E,\circ)$.

\section{Reductions for split extensions}\label{sec-reductions}

We reduce splitting over a direct product of ideals to splitting in suitable quotients, and apply this result to extensions by finite skew braces with nilpotent additive and multiplicative groups. We then consider composition series of coefficient modules and successive complements. The group-theoretic counterparts appear in \cite{Hill72}. We use the following form of Dedekind's modular law.
\begin{lemma}[Dedekind's modular law]\label{Distribute}
Let $G$ be a group, and let $X$, $Y$ and $T$ be subgroups of $G$ such that $T$ is contained in $X$. Then
\[
(X \cap Y)T = XT \cap YT.
\]
\end{lemma}

We now prove a reduction result for split brace extensions.
\begin{thm}\label{thm1}
	 Let $E$ be a skew brace with an ideal $I$ that is the internal direct product $I=I_1\times\cdots\times I_r$ of ideals $I_j$ of $E$. Here $r\geq1$ and $I_k^* = \prod_{j \neq k} I_j$, with empty sums and products of ideals interpreted as $0$. Then the following  hold true:
\begin{itemize}
	\item[(i)] $E$ splits over $I$ if and only if $E/I_k^*$ splits over $I/I_k^*$ for all $k = 1, \ldots, r$.
	\item[(ii)] If $r\geq2$, then $E$ splits over $I$ if and only if, for any distinct $k,m\in\{1,\ldots,r\}$, the quotients $E/I_k$ and $E/I_m$ split over $I/I_k$ and $I/I_m$, respectively.
\end{itemize}
\end{thm}
\begin{proof}
We first prove (i).   Let $E$  split over $I$. Then $E = G + I$ for some sub-skew brace $G$ of $E$ such that $G \cap I = \{0\}$. Then we get
\[
E / I_k^* =( G+I) / I_k^* \quad \text{for } 1 \leq k \leq r,
\]
which gives
\[
E / I_k^*   = ((G + I_k^*) / I_k^*) + (I / I_k^*).
\]
Also note that
\[
((G + I_k^*) / I_k^*) \cap (I / I_k^*) = ((G \cap I) + I_k^*) / I_k^* = I_k^* / I_k^* = \{\bar{0}\}.
\]
Hence, $E / I_k^*$  splits over $I / I_k^*$ with complement  $(G + I_k^*) / I_k^*$ for all $k = 1, \ldots, r$.

Conversely, assume  that $E / I_k^*$ splits over $I / I_k^*$ for each $k = 1, \ldots, r$. Then there exists a sub-skew brace $G_k$ of $E$ containing $I_k^*$ such that 
\[
G_k/ I_k^* +  I / I_k^* = E / I_k^*
\quad \text{and} \quad
G_k \cap I = I_k^*
\]
for each $k$.  Let $G = \bigcap_{k=1}^r G_k$, which is a sub-skew brace of $E$. Then
\[
G + I = \left( \bigcap_{k=1}^r G_k \right) + (I_1 + I_2 + \cdots + I_r)=\left( \left( \bigcap_{k=1}^r G_k \right) + I_1\right)+ (I_2 +  \cdots + I_r).
\]
Using Lemma~\ref{Distribute} and the fact that $G_k + I_k = E$ for $1 \le k \le r$, we get
\[
G + I = \bigcap_{k=1}^r (G_k  + I_k) = E.
\]
Since $G \cap I = \bigcap_{k=1}^r (G_k \cap I) = \bigcap_{k=1}^r I_k^* = \{0\}$,  $E$ splits over $I$.

We now prove (ii). Without loss of generality, let \( k = 1 \) and \( m = 2 \). Suppose that \( E \) splits over \( I \), and let \( G \) be a complement of \( I \) in \( E \). Define
\[
\overline{G}_1 := (G+I_1) / I_1 \quad \text{and} \quad \overline{I}_1 := I / I_1.
\]
We have \( (G+I_1) \cap I = I_1 \), so \( \overline{G}_1 \cap \overline{I}_1 = \{I_1/I_1\} \). Hence, \( \overline{G}_1 \) is a complement of \( \overline{I}_1 \) in \( E / I_1 \), and thus \( E / I_1 \) splits over \( I / I_1 \). A similar argument shows the result for \( I_2 \).

Conversely, suppose that \( E/I_1 \) and \( E/I_2 \) split over \( I/I_1 \) and \( I/I_2 \), respectively. Let \( \overline{G}_1 \) and \( \overline{G}_2 \) be complements of \( I/I_1 \) and \( I/I_2 \) in \( E/I_1 \) and \( E/I_2 \), respectively. As we know, \( \overline{G}_1 = G_1/I_1 \) and \( \overline{G}_2 = G_2/I_2 \) for some sub-skew braces \( G_1 \) and \( G_2 \) of \( E \) containing \( I_1 \) and \( I_2 \), respectively. Since \( \overline{G}_1 \) and \( \overline{G}_2 \) are complements of \( I/I_1 \) and \( I/I_2 \), it follows that \( G_1 \cap I = I_1 \) and \( G_2 \cap I = I_2 \). Let \( G_3 = G_1 + (I_3 + I_4 + \cdots + I_r) \), so that \( G_3 + I_2 = E \) and $G_3 \cap I_2=\{0\}$ . Define \( G = G_3 \cap G_2 \). Then we have
\[
G + I = (G_3 \cap G_2) + I = (G_3 \cap G_2) + (I_1 + I_2 + \cdots + I_r).
\]
Since $I_2\subseteq G_2$, Lemma~\ref{Distribute} gives
\[
G+I_2=(G_3+I_2)\cap G_2=E\cap G_2=G_2.
\]
Hence $G+I=G_2+I=E$.
Moreover,
\[
G \cap I = G_3 \cap G_2 \cap I = G_3 \cap I_2 = \{0\}.
\]
This shows that \( G \) is a complement of \( I \), and hence \( E \)  splits  over \( I \). The proof is complete.
\end{proof}

Let $I$ be a finite skew brace such that both $(I, +)$ and $(I, \circ)$ are nilpotent groups. Then it follows from \cite[Corollary 4.3]{CSV18} that Sylow subgroups of  $(I, +)$ are ideals of $I$. Furthermore, if $P_i$, $1 \le i \le k$, are all distinct Sylow subgroups of $(I, +)$, then $I = P_1 \oplus P_2 \oplus \cdots \oplus P_k$. With this setting, the preceding  theorem gives
\begin{cor}
Let \( E \) be a brace extension of \( H \) by \( I \), where $I$ is as in the preceding paragraph.   Define \( P_i^* = \oplus_{j \neq i} P_j \) for each \( i \). Then the following hold true:
	\begin{itemize}
		\item[(i)] \( E \) splits over \( I \) if and only if \( E / P_i^* \) splits over \( I / P_i^* \) for all \( i = 1, \ldots, k \).
		\item[(ii)] If $k\geq2$, then $E$ splits over $I$ if and only if, for any distinct $i,j\in\{1,\ldots,k\}$, the quotients $E/P_i$ and $E/P_j$ split over $I/P_i$ and $I/P_j$, respectively.
	\end{itemize}
\end{cor}
\begin{proof}
If $I=0$, the extension is split and the assertions are immediate. Otherwise, each $P_i$ is characteristic in $(I,+)$ and is also the unique Sylow subgroup for its prime in the nilpotent group $(I,\circ)$, hence characteristic there. Since $I$ is an ideal of $E$, this gives normality in both groups of $E$ and invariance under every $\lambda_a$. Thus each $P_i$ is an ideal of $E$. The assertions are now  immediate from Theorem~\ref{thm1}.
\end{proof}

\begin{remark}
Let $H$ be a skew brace. If \(M\) is an \(H\)-submodule of an $H$-module  \(I\), then \(M \times M\) is a \(\Lambda_H\)-submodule of \(I \times I\). 
\end{remark}

For an $H$-module $I$,  an  \emph{$H$-composition series} for $I$ is a series of submodules of $I$
	\[
	I = I_0 \trianglerighteq I_1 \trianglerighteq \cdots \trianglerighteq I_k = \{0\}
	\]
	such that each quotient $I_i / I_{i+1}$ is a simple $H$-module.

\begin{prop}
	Let $H$ be a skew brace and   $I$ be a finite  $H$-module.  Then there exists an $H$-composition series for $I$, and any two such composition series have the same length. Moreover, their composition factors are the same up to isomorphism and rearrangement.
\end{prop}

\begin{proof}
If $I=0$, take the series of length zero. Otherwise, the existence follows from the finiteness of $I$: if $I$ is not a simple $H$-module, then by the finiteness of $I$, there exists a maximal $H$-submodule of $I$. Repeating this process, we obtain an $H$-composition series for $I$. The uniqueness of the composition factors follows from the Jordan--H\"older theorem for groups with operators, taking the operators to be the images of the three actions.
\end{proof}

\begin{cor}\label{refine}
		Let $H$ be a skew brace and   $I$ be a finite  $H$-module. Then any series of $H$-submodules of $I$ can be refined to an $H$-composition series.
\end{cor}

\begin{prop}
Let $H$ be a skew brace and let $I$ be a finite $H$-module whose additive group is a $p$-group. Let
\[
I = I_0 \rhd I_1 \rhd \cdots \rhd I_r = \{0\}
\]
be an $H$-composition series for $I$. Then each factor $I_i / I_{i+1}$ ($0 \leq i < r$) is an elementary abelian $p$-group and a simple $H$-module. Furthermore,  if $\sigma$ is the trivial action,  then each $I_i / I_{i+1}$ is a simple $\Lambda_H$-module via the action $\phi^{(\nu, \mu)}$. 
\end{prop}

\begin{proof}
Let $K=I_i/I_{i+1}$ be any factor of the given composition series. The subgroup $pK$ is invariant under the induced actions $\nu$, $\mu$ and $\sigma$, so it is an $H$-submodule of $K$. Since $K$ is a non-zero finite abelian $p$-group, $pK<K$. Since $K$ is a simple $H$-module, we have $pK=0$, so $K$ is elementary abelian.
If $\sigma$ is trivial, Corollary~\ref{inducedaction} gives the action $\phi^{(\nu,\mu)}$ of $\Lambda_H$ on each factor $K$. Suppose that $K$ has a non-zero proper $\Lambda_H$-submodule $M$.  
In particular,
	\[
	\phi^{(\nu, \mu)}_{(h,0)}(M)=\mu^{-1}_h(M)\subseteq M
	\quad\text{and}\quad
	\phi^{(\nu, \mu)}_{(0,h)}(M)=\nu_h(M)\subseteq M,
	\]
	which contradicts the hypothesis that $K$ is a simple $H$-module. The proof is complete.
\end{proof}

We remark that if a skew brace $H$ acts on an abelian group $I$ via $(\nu, \mu, \Id)$ such that $I$ is a simple $\Lambda_H$-module via the action $\phi^{(\nu, \mu)}$, then $I$ is a simple $H$-module. 

\begin{prop}
Let $I$ be an ideal of a skew brace $E$ such that $I$ is a trivial brace and has a series
$$I = I_0 \trianglerighteq I_1 \trianglerighteq \cdots \trianglerighteq I_r = \{0\}$$
of length $r\geq1$ such that each $I_i$ is an ideal of $E$ for $0\leq i\leq r$. If there exist sub-skew braces $B_1, \ldots, B_r$ of $E$ such that $E = B_1+ I$, $B_1 \cap I = I_1$, and for each $i = 1, \ldots, r-1$, $B_i = B_{i+1} + I_i$ and $B_{i+1} \cap I_i = I_{i+1}$, then  $B_r$ is a complement of $I$ in $E$. 
\end{prop}

\begin{proof}
For $r=1$, the conclusion is part of the hypotheses, so assume $r\geq2$. It is given that $E = B_1 + I$ and $B_i = B_{i+1} + I_i$ for each $i = 1, \ldots, r-1$. We obtain
	\[
	E = B_1 + I = B_2 + I_1 + I = B_2 +  I = \cdots = B_r + I.
	\]
We claim that  $B_r \cap I = \{0\}$. Contrarily, assume that $0 \ne  x \in B_r \cap I$. Then $x \in B_i$ for each  $i = 1, \ldots, r-1$, since $B_r \subseteq B_{r-1} \subseteq \cdots \subseteq B_1$. Therefore, we get
	\[
	x \in B_1 \cap I = I_1 \Rightarrow x \in B_2 \cap I_1 = I_2 \Rightarrow \cdots \Rightarrow x \in B_{r-1} \cap I_{r-2} = I_{r-1}.
	\]
Then $x \in B_r \cap I_{r-1} = \{0\}$, which contradicts the assumption that $x \ne 0$. Hence the claim follows, and the proof is complete.
\end{proof}

\noindent {\bf Acknowledgement:} The first-named author is thankful to ISI Delhi for providing a good research ambience. He thanks HRI Prayagraj for hosting him for two weeks in September 2025, during which the work got impetus via personal interaction with the other author. He also thanks Arpan Karnar for  fruitful discussions. The second-named author acknowledges the partial support of DST-SERB Grant MTR/2021/000285.

\appendix

\section{Proofs of Theorem \ref{main7} and Theorem \ref{hallsoluble}}

\noindent{\bf Proof of Theorem \ref{main7}.}
(1) Put $N=\Lambda_I$. This is the kernel of the natural group homomorphism $\Lambda_E\to\Lambda_{E/I}$ and is a normal Hall subgroup. By the Schur--Zassenhaus theorem, $N$ has a complement $M$ in $\Lambda_E$ of order $m^2$. Let $p\colon\Lambda_E\to(E,\circ)$ be the second-coordinate projection. The orders of $\ker(p|_M)$ and $p(M)$ each divide both $m^2$ and $mn$, so each divides $m$. Their product is $m^2$, so both orders are $m$.

Set $H_2=p(M)$ and $B=\{0\}\times H_2$. Write $q\colon\Lambda_E\to\Lambda_E/N$ for the quotient map and put $K=q^{-1}(q(B))$. Since $B\cap N=\{(0,0)\}$, $B$ is a complement of $N$ in $K$. Also, $q|_M$ is an isomorphism, so $M\cap K$ is a complement of $N$ in $K$. The orders $|N|=n^2$ and $|K/N|=m$ are coprime, so at least one is odd and the corresponding group is soluble by the Feit--Thompson theorem. By the Schur--Zassenhaus theorem, the complements $B$ and $M\cap K$ of $N$ in $K$ are conjugate by an element of $N$. Conjugating $M$ by an element of $N$ that sends $M\cap K$ to $B$ gives a complement $M'$ of $N$ in $\Lambda_E$ containing $B$.

Identify $(E,+)$ with the first factor of $\Lambda_E$, and let $H_1=M'\cap(E,+)$. The same order calculation for $M'$ gives $|H_1|=m$ and $p(M')=H_2$. Both $H_1$ and $H_2$ are complements of $I$ in their respective groups by their orders. Since $H_1$ is normal in $M'$ and $(0,g)\in M'$ for every $g\in H_2$, the identity $(0,g)(h,0)(0,g)^{-1}=(\lambda_g(h),0)$ gives the required invariance. The final congruence follows from Remark~\ref{remark-ideal}.

(2) Write $z=g\circ y$ with $g\in H_2$ and $y\in I$. By Remark~\ref{remark-ideal}, $\lambda_y(h)\in h+I$. As $\lambda_g(I)=I$, we get
\[
\lambda_z(h)=\lambda_g(\lambda_y(h))\in\lambda_g(h)+I.
\]

(3) In case (i), $\lambda_x(h)-h$ lies in $I$ by Remark~\ref{remark-ideal}. Since $\lambda_x(h)$ and $h$ commute and each has additive order dividing $m$, their difference has additive order dividing both $m$ and $n$. Hence $\lambda_x(h)=h$. In case (ii), this equality follows from $x\in\Ker(\lambda)$. Every element of $E$ has the form $g\circ x$ with $g\in H_2$ and $x\in I$, and $\lambda_{g\circ x}(h)=\lambda_g(h)\in H_1$ by (1). Thus $H_1$ is a left ideal and hence a sub-skew brace. Its order and its trivial intersection with $I$ make it a complement.

(4)  Let $H_1$ be a sub-skew brace of $E$. Then $\lambda_{h}(h_1) \in H_1$ for all $h, h_1 \in H_1$.  
Since $h = g_h \circ \phi(h)$ for some $g_h \in H_2$ and $\phi(h) \in I$, we have
\[
\lambda_{h}(h_1) = \lambda_{g_h \circ \phi(h)}(h_1) = \lambda_{g_h}(\lambda_{\phi(h)}(h_1)).
\]
As $\phi(h) \in I$, we can write
\[
\lambda_{\phi(h)}(h_1) = h_1 + x, \quad \text{for some } x \in I.
\]
Thus
\begin{align*}
	\lambda_{h}(h_1) &= \lambda_{g_h}(h_1 + x) \\
	&= \lambda_{g_h}(h_1) + y, \quad \text{where } y =  \lambda_{g_h}(x) \in I.
\end{align*}
We know that  $	\lambda_{h}(h_1),  \lambda_{g_h}(h_1) \in H_1$.    Therefore,
\[
\lambda_{h}(h_1) - \lambda_{g_h}(h_1) \in H_1 \cap I =\{0\},
\]
which gives $y=0$. Since  $\lambda_{g_h}$ is an automorphism of $(E, +)$, $y=0$ implies $x=0$.  Hence $\lambda_{\phi(h)}(h_1) = h_1$.

Conversely, if $\lambda_{\phi(h)}(h_1) = h_1$ for all $h, h_1 \in H_1$, then $\lambda_{h}(h_1) \in H_1$, 
and hence $H_1$ is a sub-skew brace of $E$.

(5) The elements $h$ and $g_h$ have the same image in both quotient groups, since $h=g_h\circ\phi(h)=g_h+\lambda_{g_h}(\phi(h))$ and $\lambda_{g_h}(\phi(h))\in I$. The additive quotient map is injective on $H_1$, so the additive order of $h$ equals the order of the image of $g_h$ and divides the additive order of $g_h$. The multiplicative quotient map is injective on $H_2$, so the multiplicative order of $g_h$ equals the order of the image of $h$ and divides the multiplicative order of $h$.

(6) This follows from  Theorem~19 of \cite{CDMFT25} (see Theorem \ref{CDMFT} above). 

(7)  Let $H_1$ be the unique complement of $(I, +)$ in $(E, +)$.  Now, for any $g \in E$, consider $\lambda_g(H_1)$. Since $H_1 \cap I = \{0\}$, it follows that $\lambda_g(H_1) \cap I = \{0\}$  for all $g \in E$.  This implies that $\lambda_g(H_1)$ is also a complement of $I$ in $(E, +)$ for all $g \in E$. By the uniqueness of the complement, we have $\lambda_g(H_1) = H_1$ for all $g \in E$, which shows that $H_1$ is a left ideal of $E$. Hence $H_1$ is a sub-skew brace of $E$, and the proof is  complete.\\

\noindent{\bf Proof of Theorem \ref{hallsoluble}.} 
We proceed by induction on $|E|$. The assertion holds for $E=0$. The quotients and sub-skew braces of $E$ are soluble and inherit the commutation hypothesis. Let $N$ be a minimal non-zero ideal of $E$. By Theorem~\ref{minimumideal}, $N$ is a trivial skew brace whose additive group is elementary abelian for some prime $p$. By induction, $E/N$ has a Hall $\pi$-sub-skew brace $K/N$.

If $p\in\pi$, then $|K|=|N|\,|K/N|$ is a $\pi$-number and $|E:K|$ is a $\pi'$-number, so $K$ is a Hall $\pi$-sub-skew brace of $E$. If $p\notin\pi$, then $N$ is a Hall ideal of $K$. Theorem~\ref{main7}(3)(i) gives a complement $L$ of $N$ in $K$. Now $|L|=|K/N|$ is a $\pi$-number and $|E:L|=|E:K|\,|N|$ is a $\pi'$-number, so $L$ is a Hall $\pi$-sub-skew brace of $E$.

\end{document}